\documentclass[11pt,reqno]{amsart}

\usepackage{amsmath}
\usepackage{amssymb,amsfonts,mathrsfs, amsthm, stmaryrd}
\usepackage[colorlinks=true,linkcolor=blue,citecolor=blue]{hyperref}
\usepackage[margin=3cm,heightrounded=true,centering]{geometry}

\usepackage{euscript}
\usepackage{verbatim}
\usepackage{graphicx}

\newtheorem{maintheorem}{Theorem}

\newtheorem{theorem}{Theorem}[section]
\newtheorem{prop}[theorem]{Proposition}
\newtheorem{cor}[theorem]{Corollary}
\newtheorem{lemma}[theorem]{Lemma}

\numberwithin{equation}{section}

\DeclareMathOperator{\diag}{diag}

\DeclareMathOperator{\dvol}{dvol_g}
\DeclareMathOperator{\dvolh}{dvol_h}

\newcommand{\bR}{\mathbb{R}}
\newcommand{\bC}{\mathbb{C}}
\newcommand{\bB}{\mathbb{B}}
\newcommand{\bN}{\mathbb{N}}
\newcommand{\bS}{\mathbb{S}}

\newcommand{\cF}{\mathcal{F}}

\newcommand{\sE}{\mathcal{E}}
\newcommand{\sF}{\mathcal{F}}
\newcommand{\sH}{\mathcal{H}}
\newcommand{\sG}{\mathcal{G}}
\newcommand{\sA}{\mathcal{A}}
\newcommand{\sQ}{\mathcal{Q}}
\newcommand{\sZ}{\mathcal{Z}}
\newcommand{\sM}{\mathcal{M}}
\newcommand{\Mbar}{\overline{M}}
\newcommand{\gbar}{\overline{g}}
\newcommand{\Mdel}{\partial M}
\newcommand{\wt}{\widetilde{w}}
\newcommand{\vt}{\widetilde{v}}

\newcommand{\vep}{\varepsilon}

\renewcommand{\hbar}{\overline{h}}
\newcommand{\ghat}{\hat{g}}

\newcommand{\Rh}{\widetilde{R}}
\newcommand{\Rch}{\widetilde{Rc}}

\newcommand{\nablah}{\widetilde{\nabla}}

\begin{document}
\title[Conformally compact Ricci flow]{Ricci flow of conformally compact metrics}

\keywords{Ricci flow, conformally compact metrics, asymptotically hyperbolic metrics}

\subjclass[2000]{53C44, 58J35, 35K40, 35K59}
\author{Eric Bahuaud}
\address{Department of Mathematics,
Stanford University,
California 94305,
USA}

\email{bahuaud (at) math.stanford.edu}
\urladdr{http://math.stanford.edu/~bahuaud/}

\begin{abstract} In this paper we prove that given a smoothly conformally compact asymptotically hyperbolic metric there is a short-time solution to the Ricci flow that remains smoothly conformally compact and asymptotically hyperbolic.  We adapt recent results of Schn\"urer, Schulze and Simon to prove a stability result for conformally compact Einstein metrics sufficiently close to the hyperbolic metric.
\end{abstract}

\maketitle

\section{Introduction}

In 1989, W. X. Shi initiated the study of the Ricci flow on a noncompact manifold by proving that there is a short-time solution to the flow starting at a complete metric of bounded curvature, and moreover the flow remains in this class.  Recently there has been intense activity to understand to what extent the Ricci flow preserves other geometric conditions on noncompact manifolds, see \cite{AAR-cusp, Bamler, JMN-cusp, HQS, IMS-ac, MaXu, Woolgar} for examples.  In this paper we prove that the Ricci flow preserves the set of smoothly conformally compact asymptotically hyperbolic metrics in general dimension for a short time.  We begin by introducing these metrics.

Let $M^{n+1}$ be the interior of a compact manifold with boundary $\Mbar$.  Suppose that $x$ is a boundary defining function for $\partial M$.  This is to say that $x$ is a smooth non-negative function on $\Mbar$ that vanishes to first order precisely at $\partial M$.  We say that a metric $h$ is smoothly \textit{conformally compact} if $\hbar := x^2 h$ extends to be smooth metric on $\Mbar$.  The Poincar\'e ball model of hyperbolic space provides an easy example.

When $|dx|^2_{\hbar} = 1$ on $\partial M$, we may use $x$ to identify a collar neighbourhood of $\partial M$ in $\Mbar$ with $[0,\epsilon) \times \partial M$.  We then write $h$ as
\[ h = \frac{dx^2 + \hat{h}(x)}{x^2}, \]
for a smooth family of metrics $\hat{h}$ on $\partial M$.

If $h$ is smoothly conformally compact, with $|dx|^2_{\hbar}=1$ on $\partial M$, and $(R^{cc})_{ijkl} = h_{il} h_{jk} - h_{ik} h_{jl}$ denotes the curvature $4$-tensor of constant sectional curvature $+1$, then the curvature $4$-tensor $R$ of $h$ satisfies
\begin{align*}
 |R+R^{cc}|_h &= O(x), \; \mbox{and} \\
 |\nabla^{(j)}_h R|_h &= O(x), \; \mbox{for all} \; j. 
\end{align*}
For this reason conformally compact metrics with $|dx|^2_{\hbar} = 1$ on $\partial M$ are asymptotically hyperbolic.  It is well known that $h$ is complete and of bounded geometry.  

Recall the Ricci flow is the system of equations
\begin{equation} \label{RF} \left\{ \begin{array}{ll} \partial_{\tau} g &= - 2 Rc\; g(\tau), \\
                            g (0) &= h. \end{array} \right. \end{equation}
As previously mentioned, it follows from \cite{Shi} that there is a solution to the Ricci flow, $g(\tau)$, with initial metric $h$ for a short time.
                            
It will be more convenient to study a normalized Ricci flow.  Suppose that $g^N$ satisfies:
\begin{equation} \label{NRF} \left\{ \begin{array}{ll} \partial_{t} g^N_{ij} &= -2ng^N_{ij} - 2 Rc\; g^N_{ij}, \\
                            g^N(0) &= h. \end{array} \right. 
\end{equation}       
Setting $g(x, \tau) = (1+2n\tau) g^N( x, \frac{1}{2n} \log( 1+ 2n \tau) )$ yields a solution to the original Ricci flow.  As solutions to the Ricci flow and normalized Ricci flow differ by this time rescaling, we see that spatial regularity is preserved.  Thus a conformally compact and asymptotically hyperbolic solution to the normalized Ricci flow yields a conformally compact solution to the Ricci flow, with sectional curvatures that depend on time.  Moreover it is straightforward to check that the conformal infinity is preserved along the flow.

The first main result of this paper is the following
\begin{maintheorem} \label{theorem:main}
If $h$ is smoothly conformally compact and asymptotically hyperbolic then there exists a unique smoothly conformally compact and asymptotically hyperbolic solution $g(t)$ to \eqref{NRF} (and hence a conformally compact solution to \eqref{RF}) for a short time.
\end{maintheorem}

The proof of this Theorem proceeds as follows.  First we apply the DeTurck trick to obtain a system that may be solved by parabolic PDE techniques.  Then conditioning the equation appropriately we are able to apply a contraction mapping argument to reprove the existence (see Theorem \ref{theorem:existence}) of a short-time solution to the flow in $0$-H\"older spaces, which are H\"older spaces associated to conformally compact metrics that respect the interior geometry.  These spaces have no tangential regularity at the boundary, so the final part of the argument (see Theorem \ref{theorem:reg}) proves that the solution is smoothly conformally compact by applying regularity techniques modeled on \cite{MazzeoYamabe}.  We have stated our results for smoothly conformally compact metrics, but the arguments here extend to certain initial metrics that are polyhomogeneous.

Given short-time existence for the Ricci flow it is natural to study the stability of the flow about fixed points.  In general dimension, stability of hyperbolic space under the Ricci flow has been studied independently by Li and Yin \cite{LiYin}, Schn\"urer, Schulze and Simon \cite{SSS}, and Bamler \cite{Bamler}.  The second main result of this paper extends the stability result of Schn\"urer, Schulze and Simon to certain small Einstein perturbations of the hyperbolic metric.  By an $\eta$-admissible Einstein metric we mean an at least $C^{3,\alpha}$ conformally compact Einstein metric $h$ on $\bB^{n+1}$ that satisfies both a global curvature bound
\[ \sup_{\bB^{n+1}} |R + R^{cc}|_h \leq \eta, \]
and that the Yamabe invariant of its conformal infinity is positive.  Note the existence of such metrics follows from work of Graham and Lee \cite{GrahamLee}.  In particular any smooth Riemannian metric $\hat{h}$ on $\bS^n$ sufficiently close to the round metric in an appropriate $C^{k,\alpha}$ norm is the conformal infinity of a conformally compact Einstein metric $h$ on the unit ball.  

Before stating our second main result, we introduce the function spaces defined in \cite{SSS}.  For an interval $I \subset [0,\infty)$, let $\sM^k( \bB^{n+1} , I )$ (resp. $\sM^k_{loc}( \bB^{n+1} , I )$) denote the space of sections $g(t)$, $t \in I$ of metrics on $\bB^{n+1}$ which are $C^k$ (resp. $C^k_{loc}$) on $\bB^{n+1} \times I$, with covariant derivatives taken with respect to $h$.  The space $\sM^{\infty}_0( \bB^{n+1} , I )$ will denote metrics in $\sM^0(\bB^{n+1}, I) \cap \sM^0_{loc}(\bB^{n+1}, I)$ which are smooth for positive times and uniformly bounded in $C^k$ when restricted to time intervals of the form $[\delta, \infty)$, $\delta > 0$.

Adapting the work of Schn\"urer, Schulze and Simon we prove
\begin{maintheorem} \label{maintheorem2} Let $n \geq 3$.  There exists an $\eta(n) > 0$ so that for any $\eta$-admissible Einstein metric with $0 < \eta < \eta(n)$ the following holds.  For all $K > 0$ there exists $\epsilon_1 = \epsilon_1(n, K) > 0$ where if $g_0$ is a $\sM^0$ metric close to $h$ in the sense that  
\[ \int_{\bB^{n+1}} |g_0 - h|^2_h dvol_h \leq K, \]
and
\[ \sup_{\bB^{n+1}} |g_0 - h|_h \leq \epsilon_1. \]
Then there exists a long-time solution $g(t) \in \sM^{\infty}_0( \bB^{n+1} , [0,\infty) )$ to the normalized Ricci-DeTurck flow (with initial metric $g_0$) such that
\[ \sup_{\bB^{n+1}} |g(t) - h|_h \leq C(n, K) e^{-\frac{1}{4(n+3)} t}. \]
Moreover, $g(t) \longrightarrow h$ exponentially in $C^k$ as $t \rightarrow \infty$, for all $k \in \bN$.
\end{maintheorem}
In the above theorem, unlike in \cite{SSS}, we have transcribed the dimension to $n+1$ to match the convention of the rest of the paper.  Note also that in view of the main regularity result of this paper, if $g_0$ is smoothly conformally compact, then $g(t)$ remains smoothly conformally compact for finite time.  The limiting conformally compact Einstein metric need not be smoothly conformally compact.

We take this opportunity to mention two related papers.  First, recent work by Hu, Qing and Shi \cite{HQS} proves the Ricci flow preserves a certain class of asymptotically hyperbolic metrics for a short-time.  These metrics are defined by curvature decay conditions and, as shown in \cite{BahuaudGicquaud} and \cite{HQS}, are conformally compact of only a limited regularity.  Hu, Qing and Shi subsequently prove an interesting rigidity result.  On the other hand, in view of the applications of smoothly conformally compact metrics to geometry and physics (see for example \cite{Biquard} and references therein), it is natural to study the Ricci flow in the smooth conformally compact setting.  Second, the author and Helliwell have recently proved short-time existence results for higher-order geometric flows on compact manifolds \cite{BahuaudHelliwell}.  We observed that many short-time existence results depend only on the special algebraic structure of the flow.  Both \cite{BahuaudHelliwell} and the present paper were developed in parallel, and were inspired by recent work of Koch and Lamm \cite{KochLamm}.  The short-time existence of the Ricci flow we give here, while in the setting of conformally compact metrics, may be regarded as a concrete application of the ideas in \cite{BahuaudHelliwell}.
 
After the present paper was accepted for publication Qing, Shi and Wu \cite {QSW} posted a preprint that also studies the Ricci flow of conformally compact Einstein metrics.  Their technique differs from the one presented here in that they use maximum principles to obtain regularity along the flow.  They are also able to prove interesting perturbation results for conformally compact Einstein metrics.

This paper is structured as follows.  In Section \ref{section:prelim}, we outline the DeTurck trick and reduction of the flow to a parabolic system.  In Section \ref{section:parabolic}, we define function spaces and outline the main results from linear parabolic theory on conformally compact manifolds.  This theory is based on the edge and heat calculus for $0$-operators that appears in \cite{Mazzeo} and \cite{Albin}.  In order to not distract from the main Ricci flow argument, we have kept this section short and instead sketched several of the proofs of the analytic results in the Appendix.  In Section \ref{section:ste}, we condition the flow equations and provide the contraction mapping argument.  We discuss the regularity argument in Section \ref{section:regularity}, and the stability argument in Section \ref{section:stability}.  Finally, in the Appendix we provide sketches for the various analytic facts quoted in Section \ref{section:parabolic}.

It is a pleasure to thank Rafe Mazzeo for ideas and many helpful discussions during the course of this work.  I am also in debt to Pierre Albin, Dean Baskin, and Andr\'as Vasy for useful conversations, and Robin Graham and Fr\'ed\'eric Rochon for pointing out typos in an earlier draft.  Finally I would like to thank my collaborators in related projects, Emily Dryden, Dylan Helliwell and Boris Vertman as their input has greatly shaped my understanding and this paper.

\section{Preliminaries} \label{section:prelim}

As is well known, the Ricci flow is not a parabolic system due to the diffeomorphism invariance of the Ricci tensor.  We will break this invariance using the standard DeTurck trick.  Choosing the initial smoothly conformally compact and asymptotically hyperbolic metric $h$ as the background metric, and writing all Christoffel symbols and curvature quantities with respect this metric with tildes, we define a time dependent vector field
\[ W^k = g^{pq} \left( \Gamma^k_{pq} - \widetilde{ \Gamma^k_{pq} } \right). \]

The normalized Ricci-DeTurck flow is given by
\begin{equation} \label{NRDTF} \left\{ \begin{array}{ll} \partial_{t} g_{ij} &= -2ng_{ij} - 2 Rc\; g_{ij} + \nabla_i W_j + \nabla_j W_i, \\
                            g(0)_{ij} &= h_{ij}. \end{array} \right. 
\end{equation}  

Standard computations, for example given in \cite{Shi}, show that this flow may be written
\begin{equation} \label{NRDTF2} \left\{ \begin{array}{ll} 0 &= \partial_t g_{ij} - g^{ab} \nablah_a \nablah_b g_{ij} + 2n g_{ij} + g^{ab} g_{ip} h^{pq} \Rh_{jaqb} +  g^{ab} g_{jp} h^{pq} \Rh_{iaqb}  \\
 &- \frac{1}{2} g^{ab} g^{pq} \left( \nablah_i g_{pa} \nablah_j g_{qb} + 2 \nablah_a g_{jp} \nablah_q g_{ib} - 2 \nablah_a g_{jp} \nablah_b g_{iq} \right. \\
 &- \left. 2 \nablah_j g_{pa} \nablah_b g_{iq} - 2 \nablah_i g_{pa} \nablah_b g_{jq} \right) ,\\
g(0) &= h. \end{array} \right.
\end{equation}
From this equation we see the Ricci-DeTurck flow is a quasilinear parabolic system for the metric.

Once we prove short-time existence of a smoothly conformally compact solution $g$ to the Ricci-DeTurck flow, the time-dependent vector field $W^k$ will have coefficients smooth up to the boundary of $\Mbar$ and vanishing to first order there.  If $\phi_t$ denotes the flow generated by $W^k$, then $\ghat = \phi_t^* g$ is a solution to the normalized Ricci flow \cite{Chow}.  It is straightforward to see that $\ghat$ is smoothly conformally compact and asymptotically hyperbolic.

Finally, we only prove the existence of a short-time solution to the Ricci flow.  The uniqueness assertion in Theorem \ref{theorem:main} follows from the work of Chen and Zhu \cite{ChenZhu}. 

\section{Parabolic theory on conformally compact spaces} \label{section:parabolic}

In this section we outline linear parabolic theory for uniformly degenerate operators on conformally compact manifolds.  We just state the results we need here; sketches of proofs are deferred to the appendix.  The primary references for the material in this section are \cite{Mazzeo} and \cite{Albin}.

Let $(M,h)$ be a smoothly conformally compact asymptotically hyperbolic manifold as defined in the introduction.  Suppose that $x$ is a boundary defining function and that $\{ y^1, \cdots, y^n \}$ are coordinates on the boundary, extended to be constant in $x$.  We will refer to these coordinates as background coordinates.  Note that in order to avoid additional notation we will occasionally use the convention $x^0 = x, x^i = y^i, i = 1 \cdots n$ to generically refer to the coordinates as $x^i$ when we do not need to keep track of normal/tangential directions.  The metric $h$ decomposes as
\[ h = \frac{dx^2 + \hat{h}_{ab}(x,y) dy^a dy^b}{x^2}, \]
where the components of $\hat{h}$ are smooth up to the boundary. 

The \textit{$0$-vector fields} are generated by 
\[ \left\{ x \partial_x, x \partial_{y^1}, \cdots, x \partial_{y^n} \right\}, \]
and form the basis of a vector bundle, the $0$-tangent bundle ${}^0 TM$.  We will also have occasion to discuss \textit{$b$-vector fields}, which are generated by
\[ \left\{ x \partial_x, \partial_{y^1}, \cdots, \partial_{y^n} \right\}. \]

We will denote the space of smooth functions on $M$ by $C^{\infty}(M)$ and functions smooth up to the boundary by $C^{\infty}(\Mbar)$.  The vector bundle of symmetric $2$-tensors on $M$ will be denoted $\Sigma^2(M)$.  We will use $\frac{dx}{x}$ and $\frac{dy^b}{x}$ as the preferred basis for this bundle.

An operator $L$ on functions is \textit{uniformly degenerate} of order $m$ if in local coordinates it is given by:
\[ L = \sum_{j + |\beta| \leq m} a_{j,\beta}(x,y,t) (x \partial_x)^j (x \partial_y)^{\beta}. \]
where the coefficients $a_{j,\beta}$ are at least continuous up to the boundary.  In order to use Albin's heat calculus, we require that $a_{j, \beta}$ be smooth up to the boundary and independent of time.

The principal symbol of a uniformly degenerate operator $L$ is a homogeneous polynomial on ${}^0 T^*M$ given by
\[ {}^0 \sigma(L)( \xi, \eta) = \sum _{j + |\beta| = m} a_{j,\beta} \xi^j \eta^{\beta}. \]
We say that $L$ is \textit{elliptic} if ${}^0 \sigma(L)$ is invertible away from $(\xi, \eta) = 0$.

For the Ricci flow analysis, we will have to deal with systems of equations as our operators will act on the vector bundle of symmetric two tensors.  An operator between tensor bundles $E$ and $F$ is uniformly degenerate if in local coordinates it may be written as a system:
\[ (L u)_i = \sum_{j + |\beta| \leq m} (a_{j,\beta})_i^k (x \partial_x)^j (x \partial_y)^{\beta} u_k. \]
where the coefficients $a_{j,\beta}$ are now entries of a $\dim F \times \dim E$ matrix that is at least continuous up to the boundary.  The principal symbol is defined as before.  We will not need to consider the most general notions of ellipticity for systems as the Ricci flow system \eqref{NRDTF2} is `diagonal' at top order, i.e. $(a_{j,\beta})^{i}_k = (a_{j,\beta}) \cdot \delta^i_k$.  From this we can see that all coupling occurs at lower order.   We now say that $L$ is elliptic if $\dim F = \dim E$ and the symbol is invertible away from $(\xi,\eta) = 0$.

For the remainder of this section we suppose that $L$ is a second order uniformly degenerate elliptic operator with diagonal principal symbol.

\subsection{Function spaces} \label{section:functionspaces}

We work in the $0$-H\"older spaces defined for example in \cite{Lee, Mazzeo, MazzeoPacard}.  We describe the anisotropic version of these H\"older function spaces, and refer the reader to the references for the purely spatial version.  For any manifold $M$, the notation $M_T$ will denote the cylinder $M \times [0,T)$.  Fix a smoothly conformally compact metric $h$, which in the Ricci flow analysis, will be the initial metric.  We assume a covering of $\Mbar$ by background coordinates has been fixed.

Cover $M$ by a Whitney decomposition of countably many uniformly locally finite coordinate balls $B_i$ with centre $(x_i, y_i)$ and radius $\frac{1}{2} x_i$.  We will consider the product of each ball with a time interval $[0,T)$.  For any $0 < a < 1$, consider the norm
\begin{align*}
|| u ||_{a,\frac{a}{2}} &:= ||u||_{\infty} + \sup_{i} \left\{ \sup_{(x,y,t)\neq (x',y',t) \in (B_i)_T}  \frac{ (x+x')^{a} \left| u(x,y,t) - u(x',y',t)\right| }{|x-x'|^{a} + |y-y'|^{a}} \right. \\
& \hspace{1 in} + \left. \sup_{(x,y,t)\neq(x,y,t') \in (B_i)_T}  \frac{ \left| u(x,y,t) - u(x,y,t') \right| }{|t-t'|^{a/2}} \right\}.
\end{align*}
The prefactor $x+x'$ comes from using the euclidean metric in background coordinates instead of the intrinsic $g$-distance, see \cite{MazzeoPacard}.  Note that we may also use an affine map $\psi_i: B_T \to (B_i)_T$ from a fixed standard cylinder $B_T$ to define these norms.

Let $C^{a,\frac{a}{2}}_e(M_T)$ be the closure of $C^{\infty}(\Mbar_T)$ with respect to this norm.  We define $C^{k+a,\frac{k+a}{2}}_e(M_T)$ to consist of all functions $u$ such that $(\partial_t)^i (x \partial_x )^j (x\partial_y)^{\beta} u \in C_e^{a,\frac{a}{2}}(M_T)$ for all $2i + j + |\beta| \leq k$.  Note that unlike $C^{\infty}(\Mbar_T)$, the spaces $C^{k+a,\frac{k+a}{2}}_e(M_T)$ and even $C^{\infty, \infty}_e(M_T)$ have no tangential regularity at the boundary.  We also weight these spaces: $u \in x^{\nu} C_e^{k+a,\frac{k+a}{2}}(M_T)$ if and only if $u = x^{\nu} v$ for some $v \in C_e^{k+a,\frac{k+a}{2}}(M_T)$. 

We will also need H\"older spaces of tensors.  As previously stated, we use the vector fields $x \partial_{x}$ and $x \partial_{y^{b}}$ and covector fields $dx/x$ and $dy^{b}/{x}$ as a basis for bundles of tensors, and with this convention $\nabla^h$ involves only derivatives by the $0$-vector fields.  In this way a section of a tensor bundle is an element of a H\"older space if and only if its components are.  Furthermore, for $j \leq k$ 
\[ (\nabla^h)^j: x^{\nu} C_e^{k+a,\frac{k+a}{2}}(M_T; E) \longrightarrow x^{\nu} C_e^{a,\frac{a}{2}}(M_T; E \otimes {}^0 T^*M). \]
Finally, in what follows since we always deal with the bundle of symmetric $2$-tensors, we will not explicitly mention it in the notation.

In our final regularity argument we will need weighted H\"older spaces that allow for tangential regularity.  Following \cite{MazzeoTaylor} we introduce a scale of spaces $C^{k+a, \frac{k+a}{2}, l}(M_T)$ which consists of elements $u \in C_e^{k+a,\frac{k+a}{2}}(M_T)$ such that $\partial_y^s u \in C_e^{k-s+a,\frac{k-s+a}{2}}(M_T)$ for $0 \leq s \leq l$.  This is to say that up to $l$ of the $x \partial_y$ $0$-derivatives may be replaced by the tangential $\partial_y$ $b$-derivatives.  We weight these spaces as before.  Note that $C^{k+a, \frac{k+a}{2}, 0}(M_T) = C_e^{k+a,\frac{k+a}{2}}(M_T)$, the $0$-H\"older space, and that $C^{k+a, \frac{k+a}{2}, k}(M_T)$ is a H\"older space of $k$ $b$-derivatives.

In \cite{Lee, Mazzeo}, elliptic estimates in $0$-H\"older spaces are proved from scaling and classical interior elliptic estimates on the balls $B_i$, as the pullback of a uniformly degenerate elliptic operator under $\psi_i$ becomes uniformly elliptic.  Similarly we may obtain parabolic estimates from scaling and classical parabolic estimates.  In particular we have the following regularity result, see \cite[Theorem 8.11.1, Theorem 8.12.1]{Krylov} for the classical parabolic statements.
\begin{prop}[Parabolic regularity] \label{parabolicregularity}  Let $L$ be a second order uniformly degenerate elliptic operator.  Suppose that $D^{\gamma} a_{j,\beta} \in C^{a,\frac{a}{2}}_e(M_T)$ for $|\gamma| \leq k$, and $D^{\gamma} f \in C^{a,\frac{a}{2}}_e(M_T)$, $D^{\gamma} \phi \in C^{a}_e(M)$ for all $|\gamma| \leq k$.  If $u \in C^{2+a, \frac{2+a}{2}}_e(M_T)$ is a solution to $(\partial_t - L) u(\zeta,t) = f(\zeta,t)$ then $D^{\gamma} u \in C^{2+a,\frac{2+a}{2}}_e(M_T)$ for all $|\gamma|\leq k$.
\end{prop}

\subsection{Parabolic Schauder estimates}

We now state the main facts from linear parabolic PDE theory that we need.  We will be interested in the following problem
\begin{equation} \label{2cauchy}
 \left\{ \begin{array}{rl} 
(\partial_t - L) u(\zeta,t) & = f(\zeta,t) \\
u(\zeta,0) &= 0,\end{array} \right. 
\end{equation}
 
The basic result is 
\begin{theorem} \label{thm:schauder}
Suppose $L$ is a second order uniformly degenerate elliptic operator with time-independent coefficients.  For every $f \in x^{\mu} C^{a, \frac{a}{2}}_e(M_T)$ there is a solution $u$ to \eqref{2cauchy} in $x^{\mu} C^{2+a, \frac{2+a}{2}}_e(M_T)$.  Moreover, $u$ satisfies the parabolic Schauder estimate
\begin{equation} \label{est:schauder}
 ||u||_{x^{\mu} C^{2+a, \frac{2+a}{2}}_e(M_T)} \leq K  ||f||_{x^{\mu} C^{a, \frac{a}{2}}_e(M_T)} . 
\end{equation}
\end{theorem}
The Schauder constant $K$ that appears in the statement depends on $T$ but remains bounded as $T \rightarrow 0$.  Please see the appendix for a proof.

\subsection{Mapping properties of the heat operator}

Given the homogeneous Cauchy problem
\begin{equation} \label{cauchy-homog}
 \left\{ \begin{array}{rl} 
(\partial_t - L) u(\zeta,t) & = 0 \\
u(\zeta,0) &= \phi(\zeta),\end{array} \right. 
\end{equation}
let $A$ denote the heat operator such that takes $\phi$ to the solution of this problem, i.e. $(A\phi)(\zeta,t) = u(\zeta,t)$.  We also use the notation that $A = e^{t L}$.  In the appendix we describe how $A$ is given by an integration against a specific polyhomogeneous distribution on a certain manifold with corners that covers $M \times M \times \bR^+$.  The mapping properties of such operators follow from the asymptotics at each of the boundary hypersurfaces.  A key result that we will need is that if $V_b$ is a $b$-vector field and $A$ is a heat operator, then the commutator $[A, V_b]$ has the same asymptotics as $A$, and will enjoy the same mapping properties.  See Proposition \ref{prop:commutator} for a precise formulation.

Let $H$ denote the following time convolution of the heat operator
\[ (H f)(\zeta, t) = \int_0^t e^{(t-s)L} f(\cdot, s) ds. \]
This operator provides a solution to the inhomogeneous Cauchy problem with zero initial data.  The precise mapping properties we need are given in the following
\begin{prop} (see Corollary \ref{cor:polyhomog}) \label{prop:polyhomog} If $\phi \in x^{\mu} C^{\infty}(\Mbar)$ and $f \in x^{\mu} C^{\infty}(\Mbar_T)$ then 
\begin{enumerate}
\item $A\phi \in x^{\mu} C^{\infty}(\Mbar_T)$.
\item $Hf \in x^{\mu} C^{\infty}(\Mbar_T)$.
\end{enumerate}
\end{prop}
Once again we defer the proof to the appendix.

\section{Short-time existence} \label{section:ste}

In this section we prove short-time existence of a solution to \eqref{NRDTF2} in the $0$-H\"older spaces.  This is based on a contraction mapping argument.

We begin by making several observations that will be needed later.  Let $E = R+R^{cc}$ be the curvature `error' tensor for the conformally compact metric $h$, where $R^{cc}$ denotes the $+1$ constant curvature $4$-tensor.  By our convention for function spaces, if $h$ is smoothly conformally compact then
\[ h = \hbar_{ij} \frac{dx^i}{x} \frac{dx^j}{x} \in C^{\infty}_e(M), \]
where we recall the convention that $(x^0, x^1, \cdots, x^n) = (x, y^1, \cdots, y^n)$.

We also have $E \in x C^{\infty}(\overline{M}; T^4 M)$ and thus $E \in x C_e^{\infty}(M; T^4 M)$.

We need an expansion for the inverse of the metric.  Suppose that $v \in x C^{k+a,\frac{k+a}{2}}_e(M_T)$ with sufficiently small norm, then the symmetric $2$-tensor $h + v$ will be invertible and $(h+v)^{-1} \in C^{k+a,\frac{k+a}{2}}_e(M_T)$.  Furthermore, we document a useful expansion
\begin{equation} \label{exp-inv}
 (h+v)^{ab} = h^{ab} - h^{al} h^{bm} v_{ml} + (h+v)^{bl} h^{am} h^{pq} v_{lp} v_{mq}. 
\end{equation}

\subsection{Conditioning the Ricci-DeTurck system.}

Here we pursue short-time existence of the normalized Ricci-DeTurck flow. We will look for a solution of the form
\[ g_{ij}(x,y,t) = h_{ij}(x,y) + v_{ij}(x,y,t) \]
where $v_{ij} \in x C_e^{k+a,\frac{k+a}{2}}(M_T)$.  The system \eqref{NRDTF2} for $v$ may be written in the following way, which will facilitate treating the quasilinear system with a contraction mapping argument.  Here we handle the quasilinearity as a quadratic error.

\begin{equation} \left\{ \begin{array}{ll}  0 &= \partial_t v_{ij} - h^{ab} \nablah_a \nablah_b v_{ij} - \left((h+v)^{ab}-h^{ab}\right) \nablah_a \nablah_b v_{ij} + 2n (h+v)_{ij}\\
 & - (h+v)^{ab} (h+v)_{ip} h^{pq} \Rh_{jaqb} -  (h+v)^{ab} (h+v)_{jp} h^{pq} \Rh_{iaqb}  \\
 &+ [(h+v)^{-1} *(h+v)^{-1} * \nablah v * \nablah v]_{ij},\\
v(0) &= 0. \end{array} \right.
\end{equation}
Note that in this expression we have switched curvature sign conventions from \cite{Shi}.  Shi lowers an index in the curvature tensor to the third slot whereas I lower to the fourth slot.  The asterisk denotes linear contractions whose precise formula is unimportant for what follows.

Let us introduce notation for some of the terms above.  Define
\begin{align*}
 (T_1 v)_{ij} &:= \left((h+v)^{ab}-h^{ab}\right) \nablah_a \nablah_b v_{ij}, \\
 (T_2 v)_{ij} &:= 2n(h_{ij} + v_{ij}) + \left( - (h+v)^{ab} (h+v)_{ip} h^{pq} \Rh_{jaqb} -  (h+v)^{ab} (h+v)_{jp} h^{pq} \Rh_{iaqb} \right), \\
 (T_3 v)_{ij} &:= (h+v)^{-1} *(h+v)^{-1} * \nablah v * \nablah v. 
\end{align*}

We begin by studying the various mapping properties of the terms of this equation.  Much of the argument depends on the special algebraic structure of these equations.  We introduce the following notation.  We will say various terms are $\sQ(v)$ if they are linear combinations of contractions of bounded tensors with either $v$ or its first two $h$-covariant derivatives.  We will loosely refer to this dependence as being `quadratic', and we will make precise the estimates we need at the end of this section.  Note that indices on $\sQ$ index the term of origin in the decomposition above.

\begin{lemma} 
\[ T_1 v = \sQ_1(v), \; and \]
\[ \sQ_1 : x^{\nu} C_e^{k+a;\frac{k+a}{2}}(M_T) \longrightarrow x^{2\nu} C_e^{k-2+a;\frac{k-2+a}{2}}(M_T).\]
\end{lemma}
\begin{proof}
We begin by applying the expansion for the inverse in equation \ref{exp-inv}
\begin{align*}
 (T_1 v)_{ij} &:= \left((h+v)^{ab}-h^{ab}\right) \nablah_a \nablah_b v_{ij} \\
&= \left( h^{ab} - h^{al} h^{bm} v_{ml} + (h+v)^{bl} h^{am} h^{pq} v_{lp} v_{mq} - h^{ab} \right) \nablah_a \nablah_b v_{ij} \\
&= h^{-1} * h^{-1} * v * \nablah^2 v + (h+v)^{-1} * h^{-1} * h^{-1} * v * v * \nablah^2 v ,
\end{align*}
which shows the expression is quadratic in $v$.  Noting that $h^{-1} \in C_e^{\infty}(M)$ and $v \in x^{\nu} C_e^{k+a;\frac{k+a}{2}}(M_T)$, we see that while we lose two $0$-derivatives we gain decay in $x$, i.e. $\sQ_1 v \in x^{2\nu} C_e^{k-2+a;\frac{k-2+a}{2}}(M_T)$.
\end{proof}

The expression for $T_2$ simplifies considerably.  Note that in order to recognize the Lichnerowicz Laplacian below we will not surpress indices in the expression that follows. 

\begin{lemma} 
\[ (T_2 v)_{ij} = 2 E_{ij} + 2n v_{ij} + v_{ip} \Rch_{j}^{\;\;p} + v_{jp} \Rch_{i}^{\;\;p}+ 2v_{ml} \Rh_{\;\; ij}^{m\;\;\;l} + \sQ_2(v)_{ij}.\]
\[ \sQ_2: x^{\nu} C_e^{k+a;\frac{k+a}{2}}(M_T) \longrightarrow x^{2\nu} C_e^{k+a;\frac{k+a}{2}}(M_T). \]
\end{lemma}
\begin{proof}
By applying the expansion for the inverse to terms in $T_2$ we find the expression contains inhomogeneous terms as well as terms linear in $v$ which we must separate from the main expression.  In particular, considering one of the constituent terms in $T_2$ we find
\begin{align*}
 -(h+v)^{ab} & (h+v)_{ip} h^{pq} \Rh_{jaqb} = -(h+v)^{ab} (h+v)_{ip} \Rh_{ja \; b}^{\;\;\;\;p} \\
&= - \left(h^{ab} - h^{al} h^{bm} v_{ml} + (h+v)^{bl} h^{am} h^{pq} v_{lp} v_{mq}\right) (h_{ip} + v_{ip} ) \Rh_{ja \; b}^{\;\;\;\;p} \\
&= \Rch_{ij} + v_{ip} \Rch_{j}^{\;\;p} - v_{ml} \Rh_{\;\; ij}^{m\;\;\;l} + [h^{-1} * h^{-1} * v * v * \Rh]_{ij},
\end{align*}
where $\Rh$ in this calculation denotes the $(1,3)$ tensor.  One may check that the final quadratic contraction terms map $x^{\nu} C_e^{k+a;\frac{k+a}{2}}(M_T) \longrightarrow x^{2\nu} C_e^{k+a;\frac{k+a}{2}}(M_T)$.

Consequently,
\begin{align*}
 - (h+v)^{ab} (h+v)_{ip} h^{pq} \Rh_{jaqb} &-  (h+v)^{ab} (h+v)_{jp} h^{pq} \Rh_{iaqb} \\
&= 2 \Rch_{ij} + v_{ip} \Rch_{j}^{\;\;p} + v_{jp} \Rch_{i}^{\;\;p}- 2v_{ml} \Rh_{\;\; ij}^{m\;\;\;l} + \sQ_2(v)_{ij}
\end{align*}

Note that by the curvature asymptotics $\Rch_{ij} = -n h_{ij} + E_{ij}$ where $E_{ij} \in x C^{\infty}(\Mbar)$.  Therefore, re-assembling $T_2$ we find
\begin{align*}
 (T_2 v)_{ij} &= 2 E_{ij} + 2n v_{ij} + v_{ip} \Rch_{j}^{\;\;p} + v_{jp} \Rch_{i}^{\;\;p}- 2v_{ml} \Rh_{\;\; ij}^{m\;\;\;l} + \sQ_2(v)_{ij}
\end{align*}
\end{proof}

The third term requires no additional conditioning.
\begin{lemma} 
\[ (T_3 v)_{ij} = \sQ_3(v), \mbox{and}\]
\[ \sQ_3 : x^{\nu} C_e^{k+a;\frac{k+a}{2}}(M_T) \longrightarrow x^{2\nu} C_e^{k-1+a;\frac{k-1+a}{2}}(M_T).\]
\end{lemma}

The preceding lemmas allows us to condition the equation for $v$ further.  We now move the terms linear in $v$ to the other side of the equation.  We also have from \cite{Chow} that the term $h^{ab} \nablah_a \nablah_b v_{ij}$ is the rough Laplacian on $2$-tensors.  In fact, we see the linear elliptic part of the equation is the Lichnerowicz Laplacian on $2$-tensors,
\[ L = \Delta^h_L v_{ij} + 2n v_{ij} = h^{ab} \nablah_a \nablah_b v_{ij} + v_{ip} \Rch_{j}^{\;\;p} + v_{jp} \Rch_{i}^{\;\;p}- 2v_{ml} \Rh_{\;\; ij}^{m\;\;\;l} + 2n v_{ij}. \]

We may write:
\begin{equation}\label{eqforv} \left\{ 
\begin{array}{ll} 
\partial_t v_{ij} - (L v)_{ij} &=  \sQ v_{ij} + 2 E_{ij}, \\
                     v_{ij}(0) &= 0. 
\end{array} \right.
\end{equation}
For the remainder of the argument we drop indices.

To summarize the argument so far, we have conditioned the flow equations to recognize a strongly parabolic equation for the metric.  As the quadratic terms $\sQ$ depend on $v$ and up to its first two covariant derivatives in a polynomial fashion, there is a constant $C>0$ depending on the algebraic structure of $\sQ$ such that for all $u, v \in x^{\mu} C^{2+a,\frac{a}{2}}(M_T)$,

\begin{equation} \label{eqn:estforQ1} ||\sQ(v)||_{x^{\mu} C^{a,\frac{a}{2}}_e(M_T)} \leq C ||v||_{x^{\mu} C^{2+a,\frac{a}{2}}_e(M_T)}^2, \end{equation}
\begin{align} \label{eqn:estforQ2}
 ||\sQ(u) & - \sQ(v) ||_{x^{\mu} C^{a,\frac{a}{2}}_e(M_T)} \nonumber \\
& \leq C \max\left\{ ||u||_{x^{\mu} C^{2+a,\frac{a}{2}}_e(M_T)},||v||_{x^{\mu} C^{2+a,\frac{a}{2}}_e(M_T)} \right\} ||u - v||_{x^{\mu} C^{2+a,\frac{a}{2}}_e(M_T)}.
\end{align}

Note in these estimates that we are relaxing control of one time derivative.  This will facilitate the contraction mapping argument given in the next section.  Note also that this part of the argument will not explicitly use the gain of decay by $\sQ$.

In the regularity argument of Section \ref{section:regularity}, we will need this additional decay.  We conclude this section with the following lemma.
\begin{lemma} \label{lemma:mappingforQ}
All of the quadratic mapping terms satisfy
\[ \sQ: x^{\nu} C_e^{k+a,\frac{k+a}{2}}(M_T) \longrightarrow x^{2\nu} C_e^{k-2+a,\frac{k-2+a}{2}}(M_T). \]
Moreover, we have
\begin{itemize}
\item if $w = w' + w''$, where $w' \in x^{\nu} C_e^{k+a,\frac{k+a}{2},l}(M_T)$ and $w'' \in x^{\mu} C_e^{k+a,\frac{k+a}{2}}(M_T)$, $(\nu < \mu)$ then
\[ \sQ( w ) \in x^{2\nu} C_e^{k-2+a,\frac{k-2+a}{2},l-2}(M_T) + x^{\mu+\nu} C_e^{k-2+a,\frac{k-2+a}{2}}(M_T) \]
\item if $w = w' + w''$, where $w' \in x^{\nu} C^{\infty}(\overline{M}_T)$ and $w'' \in x^{\mu} C_e^{k+a,\frac{k+a}{2},l}(M_T)$, $(\nu < \mu)$ then
\[ \sQ( w ) \in x^{2\nu} C^{\infty}(\overline{M}_T) + x^{\mu+\nu} C_e^{k-2+a,\frac{k-2+a}{2},l-2}(M_T). \]
\end{itemize}
\end{lemma}
\begin{proof}
The first mapping property stated follows from the previous lemmas.  We need only check the final mapping properties.  These are straightforward to check as $\nablah$ acts by $0$-derivatives thus preserves the order of decay.  The explicit contraction structure of each of the terms that form $\sQ$ are:
\begin{align*} 
\sQ_1 v &= h^{-1} * h^{-1} * v * \nablah^2 v \\
\sQ_2 v &= h^{-1} * h^{-1} * v * v * \Rh \\
\sQ_3 v &= (h+v)^{-1} *(h+v)^{-1} * \nablah v * \nablah v. 
\end{align*}
Inserting $w = w' + w''$ into the expression we find the cross terms have the decay expected of $w' * w''$. 
\end{proof}

\subsection{The contraction mapping argument}

We now explain the contraction mapping argument that leads to short-time existence for equation \eqref{eqforv}.  Write the heat operator for $\partial_t - L$ as $e^{tL}$.  Apply Duhamel's principle to \eqref{eqforv} to get an equivalent integral equation
\begin{equation}
\label{eqforv2}
 v(t) = \underbrace{\int_0^t e^{(t-s)L} \left( E + \sQ(v) \right) ds }_{:= \Psi v}.
\end{equation}

Note the definition of the map $\Psi$ in the displayed equation above.

For a parameter $\mu$ and $T$ to be specified, define a subspace $\sZ_{\mu,T}$ of $x C_e^{2+a,\frac{a}{2}}(M_T)$ by
\[ \sZ_{\mu,T} = \left\{ u \in x C_e^{2+a,\frac{a}{2}}(M_T): u(x,0) = 0, ||u||_{x C_e^{2+a,\frac{a}{2}}(M_T)} \leq \mu. \right\}. \]
This is a closed subset of a Banach space.

Suppose that $u \in Z_{\mu, T}$, it follows that $v = \Psi u$ is a solution to 
\begin{equation*} \left\{ 
\begin{array}{ll} 
(\partial_t - L) v &=  \sQ(u) + E, \\
                     v(0) &= 0. 
\end{array} \right.
\end{equation*}
As $\sQ(u) + E \in x C^{a, \frac{a}{2}}(M_T)$, the Schauder estimate implies $v \in x C^{2+a,1+\frac{a}{2}}_e(M_T) \subset xC^{2+a,\frac{a}{2}}_e(M_T)$, and so
\[ \Psi: \sZ_{\mu, T} \longrightarrow x C_e^{2+a,\frac{a}{2}}(M_T). \]

We now prove that $\Psi$ in fact maps $\sZ_{\mu, T}$ to itself and is a contraction for $\mu$ and $T$ sufficiently small.  

\begin{lemma} $\Psi: Z_{\mu,T} \longrightarrow Z_{\mu,T}$ for $\mu$ and $T$ sufficiently small.
\end{lemma}
\begin{proof}
To begin, let $u \in \sZ_{\mu,T}$ and set
\begin{align*} 
v_1 &:= \int_0^t e^{(t-s)L} \sQ(u)  ds \\
v_2 &:= \int_0^t e^{(t-s)L} E ds.
\end{align*}

Consider $v_1$.  This is a solution to 
\begin{equation*} \left\{ 
\begin{array}{ll} 
(\partial_t - L) v_1 &=  \sQ( u), \\
                     v_1(0) &= 0. 
\end{array} \right.
\end{equation*}
The Schauder estimate, followed by the estimates for $\sQ$ given by equations \eqref{eqn:estforQ1} and \eqref{eqn:estforQ2} gives
\begin{align*}
||v_1||_{x C^{2+a,\frac{a}{2}}_e(M_T)} &\leq ||v_1||_{x C^{2+a,1+\frac{a}{2}}_e(M_T)} \\
&\leq K ||\sQ u||_{x C^{a,\frac{a}{2}}_e(M_T)} \\
&\leq K C ||u||^2_{x C^{2+a,\frac{a}{2}}_e(M_T)} \\
&\leq K C \mu ||u||_{x C^{2+a,\frac{a}{2}}_e(M_T)}.
\end{align*}
Taking $\mu$ sufficiently small allows us to force $K C \mu < \frac{1}{2}$.  So $||v_1||_{x C^{2+a,\frac{a}{2}}_e(M_T)} \leq \frac{\mu}{2}$.  Note that this same $\mu$ works if we shrink $T$.

Regarding $v_2$, note that this is a solution to 
\begin{equation*} \left\{ 
\begin{array}{ll} 
(\partial_t - L) v_2 &=  E, \\
                     v_2(0) &= 0. 
\end{array} \right.
\end{equation*}
We recall that $E$ and the coefficients of $L$ are smooth, time-independent and have bounded $0$-derivatives of all orders, and so by parabolic regularity any finite number of derivatives of $v_2$ are bounded.  Fixing any $\zeta$ we may write
\[ v_2(\zeta,t) = \int_0^t E(\zeta) + Lv_2(\zeta,s) ds, \; \; t \in [0,T). \]
We may now estimate the $x C^{2+a,1+\frac{a}{2}}_e(M_T)$ norm of $v_2$.  The $L^{\infty}$ norm of spatial derivatives may be controlled through the Schauder estimates by the norm of $E$, and can be made as small as we like by choosing $T$ sufficiently small.  Further, as the time derivative of $v_2$ is bounded and $v_2(x,0) = 0$, the $C^{a,\frac{a}{2}}_e(M_T)$ norm of $v_2$ can be made arbitrarily small by choosing $T$ sufficiently small.  We conclude for $T$ small enough
\[ ||v_2||_{x C^{2+a,\frac{a}{2}}(M_T)} \leq \frac{\mu}{2}. \]

Thus $\Psi: \sZ_{\mu,T} \longrightarrow \sZ_{\mu,T}$ for $t \in [0, T]$.
\end{proof}

\begin{lemma} For the $\mu$ and $T$ specified in the previous lemma, 
$ \Psi: \sZ_{\mu,T} \longrightarrow \sZ_{\mu,T} $
is a contraction.
\end{lemma}
\begin{proof}
Schauder's estimate applied to $\Psi u - \Psi v$ implies
\begin{align*}
||\Psi u &- \Psi v||_{x C^{2+a,\frac{a}{2}}_e(M_T)} \\
&\leq ||\Psi u - \Psi v||_{x C^{2+a,1+\frac{a}{2}}_e(M_T)} \\
&\leq K ||\sQ u - \sQ v||_{x C^{a,\frac{a}{2}}_e(M_T)} \\
&\leq K C \max\{ ||u||_{x C^{2+a,\frac{a}{2}}_e(M_T)}, ||v||_{x C^{2+a,\frac{a}{2}}_e(M_T)} \} ||u-v||_{x C^{2+a,\frac{a}{2}}_e(M_T)}\\
&\leq K C \mu ||u-v||_{x C^{2+a,\frac{a}{2}}_e(M_T)}.
\end{align*}

Where $K$ and $C$ are the same constants from the previous proof.  Consequently $K C \mu < \frac{1}{2}$, and $\Psi$ is a contraction.
\end{proof}

We are now ready to prove the existence of a solution to the Ricci-DeTurck flow with full $0$-regularity.
\begin{theorem} \label{theorem:existence}
If $h$ is a smoothly conformally compact metric, then there exists $T>0$ and a solution $g \in C_e^{\infty,\infty}(M_T)$ to \eqref{NRDTF2}.
\end{theorem}
\begin{proof}
The existence of a solution to $\eqref{eqforv}$ in $\sZ_{\mu, T}$ follows from the Banach fixed point theorem.  The Schauder estimate applied to the fixed point equation shows that the solution lies in $C_e^{2+a,\frac{2+a}{2}}(M_T)$. This short-time solution yields a solution in the same space to the Ricci-DeTurck flow by taking $g = h+v$.  We now improve the regularity by using a bootstrap procedure, applied to the system \eqref{NRDTF2}.  We may write this abstractly as
\[ \partial_t g + \sum_{|\beta|=0}^2 a_{\beta}(h,g) D^{\beta} g, \]
where the coefficients $a_{\beta}$ at worst satisfy $D^{\gamma} a_{\beta} \in C^{a,\frac{a}{2}}(M_T)$, for $|\gamma|=1$.  By parabolic regularity (c.f. Proposition \ref{parabolicregularity}) we conclude $D^{\gamma} g \in C_e^{2+a,\frac{2+a}{2}}(M_T)$ for all $|\gamma|=1$, which allows us to improve the spatial regularity.  By bootstrapping, and then using the equation to improve regularity in time, we find $g \in C_e^{\infty,\infty}(M_T)$.
\end{proof}

\section{Beyond $0$-regularity} \label{section:regularity} 

In the previous section we proved short-time existence of the Ricci-DeTurck flow starting at a smoothly conformally compact metric.  The solution was constructed in $0$-H\"older spaces and is smooth in time and $0$-derivatives.  We can expect more as the inhomogeneous terms in equation \eqref{eqforv2} are smooth up to the boundary with respect to background derivatives.  In this section we prove the solution remains smoothly conformally compact on the entire interval of existence.  The arguments of this section are modeled on the arguments in \cite{MazzeoYamabe}.

We begin by writing \eqref{eqforv2} more compactly as
\begin{equation} \label{eqforv3}
 v = HE + H \sQ v, 
\end{equation}
where $H$ is the time convolution of the heat operator appearing in \eqref{eqforv2}.

We begin by documenting the mapping properties of $H$ on the $C^{k+a, \frac{k+a}{2}, l}(M_T)$ spaces.
\begin{lemma}
\[ H: x^{\mu} C^{k+a, \frac{k+a}{2}, l}(M_T) \longrightarrow x^{\mu} C^{k+a, \frac{k+a}{2}, l}(M_T), l \leq k. \]
\end{lemma}
\begin{proof}
When $l = 0$, the mapping properties follow from the Schauder estimate, Theorem \ref{thm:schauder}.  The new content here is that $H$ preserves tangential regularity.

Suppose that $f \in x^{\mu} C^{k+a, \frac{k+a}{2}, 1}(M_T)$.  Taking an arbitrary $b$-derivative of $Hf$, we write
\[ \partial_y (H f) = H (\partial_y f) + [H, \partial_y] f. \]
Since $\partial_y f \in x^{\mu} C^{k-1+a, \frac{k-1+a}{2}}(M_T)$, and $H: x^{\mu} C^{k-1+a, \frac{k-1+a}{2}}(M_T) \longrightarrow x^{\mu} C^{k+1+a,\frac{k+1+a}{2}}_e(M_T)$, the first term lies again in $x^{\mu} C^{k+a,\frac{k+a}{2}}_e(M_T)$.  For the second term, by Proposition \ref{prop:commutator}, $[H, \partial_y]$ has the same mapping properties as $H$, and again maps $x^{\mu} C^{k+a,\frac{k+a}{2}}_e(M_T)$ to a subset of $x^{\mu} C^{k+a,\frac{k+a}{2}}_e(M_T)$.  This implies that $\partial_y (Hf) \in x^{\mu} C^{k+a,\frac{k+a}{2}}_e(M_T)$, and so
$H f \in C^{k+a,\frac{k+a}{2}, 1}_e(M_T)$.

The remainder of the proof proceeds by iteration.
\end{proof}

The next proposition shows that a solution to the normalized Ricci-DeTurck flow is smooth in $b$-derivatives.
\begin{prop} \label{prop:breg} Let $g$ be a solution to \eqref{NRDTF2} in $C^{k+a,\frac{k+a}{2}}_e(M_T)$, with $h$ smoothly conformally compact.  Then $g(t)$ lies in $C^{k+a, \frac{k+a}{2}, k}$ for all $t \in [0,T)$.
\end{prop}
\begin{proof}
Consider the term $HE$ in equation \eqref{eqforv3}.  As $h$ is smoothly conformally compact, $E \in x C^{\infty}(\Mbar_T)$.  Now $H$ preserves polyhomogeneity via Proposition \ref{prop:polyhomog}, and so $HE \in x C^{\infty}(\Mbar_T)$.  Thus we need only focus on the second term. 

In order to handle the term $H \sQ v$ we take advantage of the improved decay of $\sQ v$. 
If $v \in x C^{k+a,\frac{k+a}{2}}_e(M_T)$, then $\sQ v \in x^2 C^{k-2+a,\frac{k+a-2}{2}}_e(M_T)$, which $H$ maps to $x^2 C^{k+a,\frac{k+a}{2}}_e(M_T)$.  Consequently, taking any $b$-derivative of $H \sQ v$ yields
\[ \partial_y H \sQ v = x^{-1} \left(x \partial_y (H \sQ v) \right) \in x C^{k-1+a,\frac{k-1+a}{2}}_e(M_T), \]
which shows a gain of one tangential derivative so that
\[ H \sQ v \in x C^{k+a,\frac{k+a}{2}, 1}(M_T).  \]

This argument iterates $k$-times and we indicate the next step.  Apply equation \eqref{eqforv3} and Lemma \ref{lemma:mappingforQ} to obtain
\[ v = HE + H\sQ(HE + H\sQ v) \in x C^{\infty}(\Mbar_T) + x^3 C^{k+a,\frac{k+a}{2},1}(M_T), \]
and so $v \in x C^{k+a,\frac{k+a}{2}, 2}(M_T)$.
\end{proof}

We now state the main regularity result of this paper.
\begin{theorem} \label{theorem:reg} Let $g$ be a solution to \eqref{NRDTF2} in $C^{\infty,\infty}_e(M_T)$, with $h$ smoothly conformally compact.  Then $g(t)$ is smoothly conformally compact for all $t \in [0,T)$.
\end{theorem}
\begin{proof}
By Proposition \ref{prop:breg} we have that $g(t)$ is fully tangentially regular, i.e. $v \in x C^{\infty,\infty,\infty}(M_T)$.  It remains to show that $v$ is smooth up to the boundary, i.e. $v \in x C^{\infty}(\Mbar_T)$.

We will now use the structure of the heat kernel as a polyhomogeneous distribution to prove polyhomogeneity of $v$.  Given that $v \in x C^{\infty,\infty,\infty}(M_T)$ satisfies
\[ v = HE + H \sQ v, \]
We see that $\sQ v \in x^2 C^{\infty,\infty,\infty}(M_T)$, and so we may decompose $v$ as
\[ v = v' + v'' \in x C^{\infty}(\Mbar_T) + x^2 C^{\infty,\infty,\infty}(M_T). \]
We now insert this back into \eqref{eqforv3}.  Using Lemma \ref{lemma:mappingforQ}, we find
\[ \sQ( v' + v'' ) = x^2 C^{\infty}(\Mbar_T) + x^3 C^{\infty,\infty,\infty}(M_T). \]
Equation \eqref{eqforv3} now lets us conclude
\[ v \in x C^{\infty}(\Mbar_T) + x^3 C^{\infty,\infty,\infty}(M_T). \]
Iterating we conclude $v \in x C^{\infty}(\Mbar_T)$.

Finally, we have proved $v \in x C^{\infty}(\Mbar_T)$, i.e.
that 
\[ v = x \overline{v}_{ij} \frac{dx^i}{x} \frac{dx^j}{x}, \]
where $\overline{v}_{ij}$ is smooth up to the boundary, and we remind the reader of the convention $(x^0, x^1, \cdots, x^n) = (x, y^1, \cdots, y^n)$.  So now $x^2 v = x \overline{v}_{ij}dx^i dx^j$ and consequently $g = h+v$ is smoothly conformally compact. 
\end{proof} 

We note that $\gbar = x^2 g = (\overline{h}_{ij} + x \overline{v}_{ij} ) dx^i dx^j$, so that $|dx|^2_{\gbar} = 1$ on $\partial M$, so that $g$ is asymptotically hyperbolic.  Combining Theorem \ref{theorem:existence} and Theorem \ref{theorem:reg} completes the proof of Theorem \ref{theorem:main}. 

We conclude this section by remarking that it is possible to relax the condition that the initial metric $h$ be smoothly conformally compact.  Our entire argument applies to certain polyhomogeneous initial metrics as well.  The heat kernel analysis from the Appendix and the heat kernel mapping properties extend to this case in a straightforward manner.

\section{Stability about admissible Einstein metrics} \label{section:stability}

In this section we adapt the arguments from \cite{SSS} to prove stability of the normalized Ricci-DeTurck flow near $\eta$-admissible Einstein metrics.  

The main idea is to replace the background hyperbolic metric on the ball used in \cite{SSS} with an $\eta$-admissible Einstein metric, $h$.  The existence arguments in \cite[Theorem 2.4]{SSS} works with a such a complete Einstein metric in place of the hyperbolic metric.  In order to obtain the long-time existence and convergence result, we must adapt the $L^2$ estimate for the difference of the solution of the flow with the background metric due to the `curvature error' terms, whose magnitude is measured by $\eta$, that arise.  There is sufficient room in the main estimate to allow for metrics with $\eta$ sufficiently small.  To this end, we replace \cite[Lemma 2.2]{SSS} with the following

\begin{lemma} \label{lemma:normest}
Suppose that $g \in \sM^{\infty}(\bB^{n+1}, (0,T) )$ is a solution to the normalized Ricci-DeTurck flow which is $\varepsilon$-close to $h$, where $h$ is an $\eta$-admissible Einstein metric and $\vep$ is sufficiently small.  Then
\begin{equation}
\partial_t |g-h|^2 \leq g^{ij} \nablah_i \nablah_j |g-h|^2 - (2-\varepsilon) |\nablah (g-h)|^2 + (4 + \varepsilon + b(n) \eta) |g-h|^2,
\end{equation}
where $b(n)$ is a constant depending on $n$, and where any constant $c(n) \varepsilon$ is replaced with $\varepsilon$, and norms and covariant derivatives are with respect to $h$.
\end{lemma}
\begin{proof}
Choose $h$-normal coordinates at any point $p$ and write $h_{ij} = \delta_{ij}$, then diagonalize $g$ at $p$ to write $g_{ij} = \lambda_i g_{ij}$, where $\lambda_i > 0$.  Note that we drop Einstein summation convention as we use this expression in the remainder of the proof.

Note that for a metric $g$ $\vep$-close to $h$, the eigenvalues of $g$ with respect to $h$ satisfy $(1+\vep)^{-1} \leq \lambda_i \leq 1+\vep$. 

Set $Z = g - h$.  We compute using equation \eqref{NRDTF2}
\begin{align*}
\partial_t |Z|^2 &= 2 \sum_i (g_{ii} - h_{ii}) \partial_t g_{ii} \\
&= 2 \sum_i (g_{ii}-h_{ii}) \left( g^{ab} \nablah_a \nablah_b g_{ii} + 2 g^{ab} g_{ip} h^{pq} E_{iaqb} \right. \\
& \left. + 2 (h_{ii} - g_{ii}) + 2 g_{ii} g^{ab} (h_{ab} - g_{ab}) 
+ [g^{-1} * g^{-1} * \nablah g * \nablah g]_{ii} \right),
\end{align*}
where $E$ represents the curvature deviation from hyperbolic space: $E = R + R^{cc}$.

We now estimate exactly as in \cite{SSS}.  The only new ingredient are the curvature error terms.  Given an index $i$, we find the expression for $2 \sum_i (g_{ii} - h_{ii}) \partial_t g_{ii}$ contains
\begin{align*}
2 \sum_i 2 (g_{ii} - h_{ii}) g^{ab} g_{ip} h^{pq} E_{iaqb} &= 4 \sum_{i,q} (g_{ii} - h_{ii}) g^{ab} g_{iq} E_{iaqb}\\
&= 4 \sum_i (\lambda_i - 1) \lambda_i g^{ab} E_{iaib} \\
&= -4 \sum_i (\lambda_i - 1) \lambda_i g^{ab} E_{a i i b} \\
&= -4 \sum_i (\lambda_i - 1) \lambda_i \sum_j \lambda_j^{-1} E_{j i i j} \\
&= -4 \sum_i (\lambda_i - 1) \lambda_i \sum_j (\lambda_j^{-1} - 1 )E_{j i i j} 
\end{align*}
where in the last line we used the fact that since $h$ is Einstein $\sum_j E_{ijji} = 0$.  Noting from symmetries of the curvature tensor that $E_{iiii} = 0$, we find
\begin{align*}
-4 \sum_i (\lambda_i - 1) \lambda_i \sum_j (\lambda_j^{-1} - 1 )E_{j i i j} 
&= 4 \sum_{i \neq j} \lambda_i \lambda_j^{-1} (\lambda_i - 1) (\lambda_j - 1) E_{i j j i} 
\end{align*}
Estimating using the $h$-norm yields
\begin{align*}
\left|-4 \sum_i (\lambda_i - 1) \lambda_i \sum_j (\lambda_j^{-1} - 1 )E_{j i i j} \right|_h
&= \left|4 \sum_{i \neq j} \lambda_i \lambda_j^{-1} (\lambda_i - 1) (\lambda_j - 1) E_{i j j i} \right|_h \\
& \leq \left( \max_{i,j} |E_{ijji}|_h \right) b(n) |Z|^2 \\
& \leq \eta b(n) |Z|^2.
\end{align*}
\end{proof}

In order to proceed we must replace the McKean inequality used in \cite{SSS} for the infinimum of the $L^2$ spectrum of the hyperbolic Laplacian on functions with its counterpart for a conformally compact Einstein metric.  By a result of Lee \cite{LeeSpectrum}, if $h$ is at least $C^{3,\alpha}$ conformally compact with smooth conformal infinity, and the conformal infinity has positive Yamabe invariant, then $\lambda_0(h) = n^2/4$. 

We now obtain $L^2$ control of $|g(t)-h|$ by modifying \cite[Theorem 3.1]{SSS}.  Let $\eta(n) = \frac{1}{8 b(n)}$, where $b(n)$ is the constant appearing in the previous lemma.  We assume that for $0 < \eta < \eta(n)$ that we have a fixed $\eta$-admissible Einstein metric.

\begin{theorem} \label{theorem:SSSL2est}
Let $n \geq 3$.  There exists $\delta_0 = \delta_0(n) > 0$ such that if $g \in \sM^{\infty}(B_R, [0,T))$ is a solution to the normalized Ricci-DeTurck flow with $g = h$ on $\partial B_R(0) \times [0,T)$ and $\sup_{B_R(0) \times [0,T)} |g-h| \leq \delta_0$, then
\[ \int_{B_R(0)} |g(t) - h|^2_h \dvolh \leq e^{-at} \int_{B_R(0)} |g(0) - h|^2_h \dvolh. \]
$a \geq 1/4$.
\end{theorem}
\begin{proof}
Assume that $\delta_0$ is so small that $g$ is $\vep$-close to $h$.

We compute using the Lemma \ref{lemma:normest},
\begin{align*}
\partial_t \int_{B_R(0)} |Z|^2 \dvolh & \leq \int_{B_R(0)} g^{ij} \nablah_i \nablah_j |Z|^2 - (2-\varepsilon) |\nablah Z|^2 + (4 + \varepsilon + \eta) |Z|^2  \dvolh.
\end{align*}

After integrating by parts and estimating as in \cite{SSS}, we obtain

\begin{align*}
\partial_t \int_{B_R(0)} |Z|^2 \dvolh & \leq \int_{B_R(0)} - (2-\varepsilon) |\nablah Z|^2 + (4 + \varepsilon + b(n) \eta) |Z|^2  \dvolh,
\end{align*}

We now use the eigenvalue estimates described above.  Kato's inequality implies that $| \nablah |Z| |^2 \leq |\nablah Z|^2$.  Further,
\begin{align*}
\int_{B_R(0)} |\nablah |Z||^2 \dvolh \geq \lambda_0(B_R(0)) \int_{B_R(0)} |Z|^2 \dvolh
\end{align*}
where $\lambda_0(B_R(0)) \geq \lambda_0(h) = \frac{n^2}{4}$.
So now,
\begin{align*}
\int_{B_R(0)} - (2-\varepsilon) |\nablah Z|^2 & \leq \int_{B_R(0)} - (2-\varepsilon) | \nablah |Z| |^2 \dvolh \\
&\leq \int_{B_R(0)} - (2-\varepsilon) \frac{n^2}{4} |Z|^2 \dvolh \\
&= \left(-\frac{n^2}{2} + \vep \right) \int_{B_R(0)} |Z|^2 \dvolh.
\end{align*}
Finally, we obtain
\[ \partial_t F \leq \left(-\frac{n^2}{2} + \vep + 4 + b(n) \eta \right) F, \]
where $F = \int_{B_R(0)} |Z|^2 \dvolh$.

We must have $-\frac{n^2}{2} + \vep + 4 + b(n) \eta < 0$ in order to get an exponential decay estimate.  Thus we can take $\vep$ so small that (say) $\vep < 1/8$.  Since $b(n) \eta < 1/8$ and $n \geq 3$, we see $-\frac{n^2}{2} + \vep + 4 + b(n) \eta \leq -\frac{1}{2} + \frac{1}{4} = -\frac{1}{4}$.
\end{proof}

We now have all of the new ingredients needed to prove Theorem \ref{maintheorem2}.  Let $0 < \eta < \eta(n)$ and $h$ be an $\eta$-admissible Einstein metric.  We proceed exactly as in \cite[Theorem 3.4]{SSS}. 

\appendix
\section{More on Linear Parabolic PDE theory}

In this appendix we give more detail surrounding linear parabolic theory on conformally compact asymptotically hyperbolic manifolds.  Our approach to understanding these operators is based on the edge heat calculus developed in \cite{Albin}.  Note that in this appendix we deal exclusively with the $0$-case but the arguments generalize in a straightforward manner to the full complete edge case.

The point of view we adopt is that for a second order uniformly degenerate elliptic operator with time-independent coefficients, we can explicitly construct the heat kernel as a polyhomogeneous distribution on an appropriate manifold with corners that covers $M \times M \times \bR^+$.  In this section we will first describe this blow up space.  We then proceed to discuss the heat kernel as constructed in \cite{Albin}.  We then prove several mapping properties of these kernels.  We conclude by proving Schauder type estimates.

We now introduce the appropriate blow up spaces for the construction of the heat kernel.  First we define the 0-double space: $M^2_e$, originally introduced in \cite{Mazzeo} for the elliptic edge calculus.  This is a manifold with corners that covers $M^2$, and is obtained by introducing polar coordinates around the submanifold
\[ \Mdel \times_B \Mdel = \{ (w, w') \in \Mdel \times \Mdel: w = w' \}. \]
So $M^2_e = [M \times M; \Mdel \times_B \Mdel]$.  This will introduce three new boundary hypersurfaces; following Albin we denote these by $B_{11}$ (the front face), $B_{01}$ (the right boundary) and $B_{10}$ (the left boundary).  We denote the blowdown map
\[ \beta_e: M^2_e \rightarrow M^2, \]
and the edge diagonal by
\[ \diag_e = \overline{ \beta^{-1}_e( \diag \backslash \Mdel \times_B \Mdel) }. \]

We describe the edge double space in terms of coordinate charts.  In the interior of $M^2_e$ we may use the usual coordinates
\[ \left( (x,y), (x',y') \right) = \left( \zeta, \zeta' \right), \]
where $y$ will always denote coordinates along $B$ and $z$ will always denote coordinates along $F$.  We will favour the following projective coordinates for $M^2_e$, defined away from $B_{10}$ and that express the edge diagonal easily are given by 
\[ \left( (x,y,z), \left( s := \frac{x'}{x}, v := \frac{y' - y}{x}\right) \right). \]

Note that in these coordinates, $s=0$ is a defining function for $B_{01}$ and $x=0$ for the front face (away from $B_{10}$).  By reversing the roles of $x$ and $x'$ in the obvious manner, one may obtain a second chart covering the remainder of $M^2_e$.

We now introduce the heat space $HM^2_e$.  This is given by a parabolic blow up of the manifold $M^2_e \times \bR_{+}$ along the submanifold $\diag_e \times \{0\}$.  This gives us a number of new boundary hypersurfaces.  We keep Albin's notation for these, illustrated in Figure~\ref{blowup}.

\begin{figure}[h]
\includegraphics[scale=0.5]{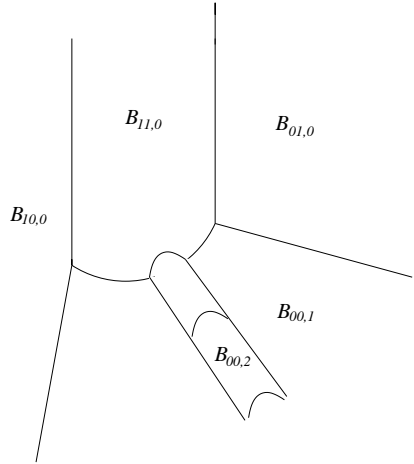}
\caption{The blown up heat space, $HM^2_e$}
\label{blowup}
\end{figure}

We now discuss the coordinate systems we can use on $HM^2_e$.  In what follows we work away from $B_{10,0}$ (i.e. away from $x=0$).  Near $B_{11,0}$, and away from $B_{00,1}$ we can use
\begin{equation} \label{coord-near-lc} \left( (x,y), \left( s' := \frac{x'}{x}, v' := \frac{y' - y}{x}\right), \tau := t^{1/2} \right). 
\end{equation}

Near $B_{11,0}$ and the `top' of $B_{00,2}$ we may use
\begin{equation} \label{coord-near-tff} \left( (S,U), \zeta', \tau \right) := \left( \left( \frac{x-x'}{x' t^{1/2}}, \frac{y-y'}{x' t^{1/2}} \right), (x',y'), t^{1/2} \right). 
\end{equation}
In these coordinates, $x'=0$ defines the $B_{11,0}$ and $\tau = 0$ defines $B_{00,2}$.  Finally, near $B_{11,0}$ and the `bottom' of $B_{00,2}$, close to $B_{00,1}$ we appear to need to introduce another coordinate system.  However, we observe that this region is reached using the above coordinates as $|(S,U)| \rightarrow +\infty$.  We will soon see that our heat kernels vanish to infinite order along this boundary.

We will denote the full blow down map $\beta: HM^2_e \rightarrow M^2 \times \bR^+$.  

Given a manifold with corners $M$, ${C}^{\infty}(\Mbar)$ denotes functions on $M$ that are smooth in the interior and smooth up to all boundary hypersurfaces.  The space $\dot{C}^{\infty}(M)$ will denote smooth functions vanishing to all orders at the boundary hypersurfaces.  If $\cF$ denotes a list of boundary hypersurfaces then $\dot{C}^{\infty}_{\cF}(M)$ denotes smooth functions vanishing to all orders at all boundary hypersurfaces except those in $\cF$; at the other hypersurfaces we demand the functions are smooth up to the boundary. 

We will also need to define sets of functions that have asymptotic expansions at the boundary hypersurfaces.  Let $M$ be a manifold with corners with boundary defining functions $x_i$.  A distribution $u$ is polyhomogeneous conormal\footnote{See \cite{Grieser} for a discussion and to make the meaning of $\sim$ precise.} if:
\[ u \sim \sum_{Re s_j \rightarrow \infty} \sum_{p=0}^{p_j} a_{j,p}(x,y) x^{s_j} (\log x)^p, \]
where $a_{j,p} \in C^{\infty}(\Mbar)$.  We'll denote the set of such distributions $\mathcal{A}^*_{phg}$.  We can also restrict the set of exponents that may occur above.  Define an index set to be a discrete subset $E \subset \bC \times \bN_0$ such that
\begin{enumerate}
\item if $(s_j, p_j) \in E$ and $|(s_j,p_j)| \longrightarrow \infty$, then $Re(s_j) \longrightarrow \infty$.\\
\item if $(s,p) \in E$ then $(s+k, p-l) \in E$ for any $k, l \in \bN, l \leq p$.
\end{enumerate}
Given a set of index sets $\mathcal{E}$ for each boundary hypersurface, we denote by $\sA^{\mathcal{E}}_{phg}$ the set of polyhomogeneous conormal functions with exponents ranging in $\mathcal{E}$.  Note that we will use a few special notations for index sets.  The empty set will denote the index set for a function vanishing to all orders along a hypersurface.  A single number $n \in \bN_0$ will denote the index set $\{ (j, 0): j\in \bN, j \geq n \}$ of functions vanishing to order $n$.  Note that the index set $\{ 0 \}$ represents functions smooth up to the hypersurface.  For more details about operations on these sets, see the concise review in \cite[Appendix A]{Mazzeo}.

\subsection{The heat kernel of a uniformly degenerate elliptic operator}

Let $L$ be a second order uniformly degenerate elliptic operator.  We consider a heat type equation
\[ \left\{ \begin{array}{rl} 
(\partial_t - L) u(\zeta,t) & = 0 \\
u(\zeta,0) &= f(\zeta),\end{array} \right. \]
where $f \in \Gamma(M;\sE)$ is a smooth section of a vector bundle $\sE$.

The heat kernel of $L$ is a distribution on $M^2 \times \bR^+$ so that the solution to the above problem is given by:
\[ u(\zeta, t) = \int_M h(\zeta, \zeta', t)f(\zeta') \dvol(\zeta'). \]
Here $h$ formally satisfies:
\begin{equation} \label{heatkerneqn} \left\{ \begin{array}{rl} 
(\partial_t - L_{\zeta}) h(\zeta,\zeta',t) & = 0 \\
h(\zeta,\zeta',0) &= \delta(\zeta-\zeta'),\end{array} \right. 
\end{equation}
We will see that $h = \beta_* H$, where $H$ is a polyhomogeneous distribution on $HM^2_e$. 

The actual construction of this distribution is done for half-densities, so that it makes sense to compose operators.  We briefly review Albin's construction of the heat calculus.  We define a weighted bundle of half-densities $D := \rho_{00,2}^{-\frac{n}{2} + 2} \rho_{11,0}^{-\frac{n+1}{2}} \Omega^{1/2}(HM_e^2)$.  Kernels of operators in the heat calculus are elements of
\[ K^{k, l}(M, D) := \rho_{00,2}^k \rho_{11,0}^l \dot{C}^{\infty}_{B_{00,2}, B_{11,0}}(HM_e^2; D). \]
The action of a kernel $K_A$ in $K^{k,l}$ on smooth half-densities is given by
\[ A(f)(\zeta,t) = \int_M \beta_* K_A( \zeta, \zeta', t) f(\zeta'). \]
We'll denote the operator $A$ acting in this manner by $A \in \Psi_{e, Heat}^{k,l}$.

Albin proves:
\begin{theorem} If $L$ is the scalar Laplacian of a exact edge metric, then $A \in \Psi^{2,0}_{e, Heat}$, where $A$ is the heat operator of $\partial_t - L$.
\end{theorem}

We note that Albin's construction is closely modeled on the work of Melrose \cite{Melrose}, and generalizes in a straightforward manner to general second order uniformly degenerate elliptic operators.  Furthermore, Melrose also considers the case of elliptic operators between bundles \cite[Theorem 7.29]{Melrose} with diagonal principal symbol.  Thus we have

\begin{theorem} \label{thm-htkrnl-strc} If $L$ is a uniformly degenerate elliptic operator with diagonal principal symbol, then $A \in \Psi^{2,0}_{e, Heat}$, where $A$ is the heat operator of $\partial_t - L$.
\end{theorem}

We now give a brief indication of the proof of theorem and refer the reader to \cite{Albin, Melrose} for further detail.  We work in $HM^2_e$ with the ansatz that the solution already vanishes to infinite order at $B_{10,0}, B_{01,0},$ and $B_{00,1}$.  To deal with the rest of the equation and boundary hypersurfaces involves three main steps.  First, an initial parametrix is constructed by pulling the heat equation back to $HM^2_e$ in coordinates near the blown up diagonal.  As $B_{00,2}$ fibres over the diagonal, we find that the equation restricts to a Euclidean type heat equation on each fibre with smooth coefficients in the variables along the fibre.  Thus we may progressively solve away the Taylor series at $B_{00,2}$ with control of the asymptotics down to $B_{11,0}$.  This handles the initial condition.  The second step is to progressively solve away the Taylor series at $B_{11,0}$ using the heat kernel of hyperbolic space (recall $0$-metrics are asymptotically hyperbolic).  The result of these two steps is a parametrix solving the heat equation to infinite order at all boundary hypersurfaces.  To improve the parametrix to an actual inverse requires an argument involving Volterra operators and is given in \cite[Proposition 7.17]{Melrose}.

Finally we note that the construction above also works when the background metric is polyhomogeneous.

\subsection{Mapping properties}

In this section we study the action of the heat kernels in $\Psi^{2,0}_{e, Heat}$ above on 
functions, using Melrose's pushforward theorem.  Figure \ref{fig1} introduces some important notation.

\begin{figure}[h]
  \includegraphics[scale=0.75]{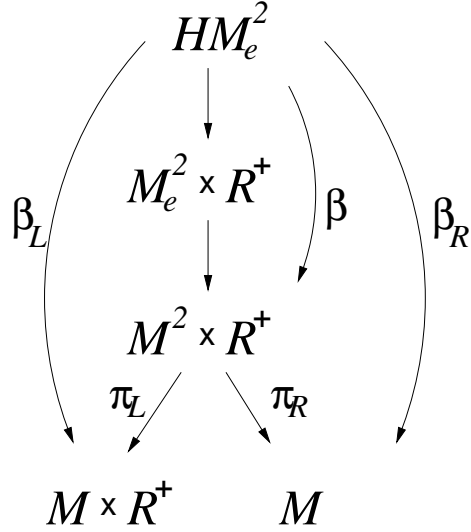}
\caption{Definition of various maps}
\label{fig1}
\end{figure}

We identify functions and half-densities on $M^2 \times \bR^+$ and the factors $M \times \bR^+$ and $M$ by\footnote{Here we omit the smooth factor $\sqrt{\det g}$ in the densities that follow.}
\[ f(x,y,x',y',t) \leftrightarrow f(x,y,x',y',t) x^{-\frac{(n+1)}{2}} (x')^{-\frac{(n+1)}{2}} |dx dy dx' dy' dt|^{1/2}, \]
\[ f(x,y,t) \leftrightarrow f(x,y,t) x^{-\frac{(n+1)}{2}} |dx dy dt|^{1/2}, \]
\[ f(x,y) \leftrightarrow f(x) x^{-\frac{(n+1)}{2}} |dx dy|^{1/2}. \]

From \cite[page 13]{Albin}  an element of $A \in \Psi^{2,0}_{e, Heat}$ has an integral kernel that may be written as $\rho_{00,2}^{-\frac{n}{2}} \rho_{11,0}^{-\frac{n+1}{2}} k \cdot \nu$, where $k$ is a function that vanishes to infinite order at $B_{10,0}, B_{01,0}, \mbox{and} \; B_{00,1}$, and is smooth up to the boundary at $B_{00,2}$ and $B_{11,0}$, and $\nu$ is a smooth section of $\Omega^{1/2}(HM^2_e)$.

An operator $A \in \Psi^{2,0}_{e, Heat}$ acts on half-densities by
\begin{equation} \label{eqn:actionhalf} (A f)(x,y,t) x^{-\frac{n+1}{2}} |dx dy dt|^{1/2} = (\beta_L)_*\left( \rho_{00,2}^{-\frac{n}{2}} \rho_{11,0}^{-\frac{n+1}{2}} k \nu \cdot (\beta_R)^*( f(x',y') (x')^{-\frac{n+1}{2}} |dx' dy'|^{1/2} ) \right). 
\end{equation}

To relate these half-densities, let us work in coordinates near $B_{11,0}$ and $B_{00,2}$.  We may take $\nu = |dS dU dx' dy' d\tau|^{1/2}$.  Pulling our standard half-density on $M^2 \times \bR^+$ back we find
\begin{align*} \beta^* &( x^{-\frac{n+1}{2}} (x')^{-\frac{n+1}{2}} |dx dy dx' dy' dt|^{1/2} ) \\
&= (1+S\tau)^{-\frac{n+1}{2}} (x')^{-\frac{n+1}{2}} (x')^{-\frac{n+1}{2}} ( 2 (x')^{n+1} \tau^{n+2} )^{1/2} |dS dU dx' dy' d\tau|^{1/2} \\
&= \sqrt{2} (1+S\tau)^{-\frac{n+1}{2}} (x')^{-\frac{n+1}{2}} \tau^{\frac{n+2}{2}} \nu. 
\end{align*}
The factor $\sqrt{2} (1+S\tau)^{-(n+1)/2}$ is smooth and uniformly bounded, so we omit it hereafter.  

In order to apply Melrose's push-forward theorem, we must work with smooth b-densities.  Here is how to arrange this.  We multiply both sides by the half density $x^{-\frac{n+1}{2}} |dx dy dt|^{1/2}$, and noting
\begin{align*}
\beta_L^* & (x^{-\frac{n+1}{2}} |dx dy dt|^{1/2}) \beta_R^* ((x')^{-\frac{n+1}{2}} |dx' dy'|^{1/2} ) = \beta^*( x^{-\frac{n+1}{2}} (x')^{-\frac{n+1}{2}} |dx dy dx' dy' dt|^{1/2} ),
\end{align*}
we find the action on smooth densities is given by
\begin{align*}
(A f)&(x,y,t)x^{-(n+1)} |dxdydt| = (\beta_L)_*\left( \rho_{00,2}^1 \rho_{11,0}^{-n-1} k \cdot \beta_R^* f \; \nu^2 \right). 
\end{align*}

Finally we introduce a total defining function on both sides of this equation to obtain $b$-densities, which we denote with a subscript $b$.  In this case we have
$\nu^2_b = (\rho_{10,0} \rho_{01,0} \rho_{11,0} \rho_{00,1} \rho_{00,2})^{-1} \nu^2$, and $\sigma_b := \frac{1}{xt} |dx dy dt|$.  Now
\begin{align*} 
(A f)&(x,y,t) x t x^{-(n+1)} \; \sigma_b  \\
&= (\beta_L)_*\left( \rho_{10,0} \rho_{11,0}^{-n} \rho_{01,1} \rho_{00,1} \rho_{00,2}^{2} k \cdot(\beta_R)^*( f(x',y') )\; \nu_b^2 \right)
\end{align*}

We now apply this to the following
\begin{prop} Let $A \in \Psi^{2,0}_{e, Heat}$.  If $f \in \sA^{\sF}_{phg}(M)$ then $Af \in \sA^{(\sF, 0)}_{phg}(M \times \bR^+)$.
\end{prop}
\begin{proof}
First, let us consider the b-map $\beta_R$.  As no boundary hypersurface is mapped to a corner of $M$, $\beta_R$ is a b-fibration.  It is easy to check that the exponent matrix for this map is 
\begin{center}
\begin{tabular}{ l l l l l l }
  & $B_{10,0}$ & $B_{11,0}$ & $B_{01,0}$ & $B_{00,1}$ & $B_{00,2}$ \\
  $\partial X$ & 1 & 1 & 0 & 0 & 0 
\end{tabular}
\end{center}
As a consequence, if $f \in \sA^{\sF}_{phg}(M)$ with index set $\sF$, $\beta_R^* f \in \sA^{\{ \sF, \sF, 0, 0, 0 \}}_{phg}(HM^2_e)$, by the pull-back theorem \cite[Proposition A.13]{Mazzeo}.

The function $k$ is polyhomogeneous with respect to the index set $\{ \emptyset, 0, \emptyset, \emptyset, 0 \}$.  Accounting for the powers of the defining functions we obtain the index set $\{ \emptyset, -n, \emptyset, \emptyset, 2 \}$.  The index set for the product of this expression with the pull-back of $f$ is then $\mathcal{G} = \{ \emptyset, -n + F, \emptyset, \emptyset, 2. \}$.

The map $\beta_L$ is a b-fibration.  The exponent matrix for this map is 
\begin{center}
\begin{tabular}{ l l l l l l }
  & $B_{10,0}$ & $B_{11,0}$ & $B_{01,0}$ & $B_{00,1}$ & $B_{00,2}$ \\
  $\sH_1$ & 0 & 1 & 1 & 0 & 0 \\
  $\sH_2$ & 0 & 0 & 0 & 1 & 1
\end{tabular}
\end{center}

In the above table we have the labeled hypersurfaces of $M \times \bR^+$ in the following manner: $\sH_1$ represents $t=0$ and $\sH_2$ represents $x=0$.  We now apply Melrose's pushfoward theorem \cite[Proposition A.18]{Mazzeo}.  Note that the integrability condition is met at $B_{10,0}$ as $Re(\sG(B_{10,0})) > 0$.  Now the index set for $\sH_1 = \sG( B_{11,0} ) \overline{\cup} \sG(B_{01,0}) = -n+\sF$ and $\sH_2 = \sG( B_{00,1} ) \overline{\cup} \sG(B_{00,2}) = 2$.  Note that $\sH_2 = 2$ is not surprising since $\tau^2 = t$, from the parabolic blow up.

The calculation here provides the index sets needed to for computing the asymptotics of \[ (A f)(x,y,t) x t x^{-(n+1)} \; \sigma_b. \]
Canceling the powers of the defining functions, and returning to the identification of densities with functions now shows that $Af \in \sA^{(\sF, 0)}_{phg}(M \times \bR^+)$. 
\end{proof}

We use the above proposition primarily in the form
\begin{cor} \label{cor:polyhomog} If $f \in x^{\mu} C^{\infty}(\Mbar)$ then $Af \in x^{\mu} C^{\infty}(\Mbar_T)$.
\end{cor}
\begin{proof}
The only point that we have to be careful about is that the previous theorem only guarantees an expansion in powers of $\tau = \sqrt{t}$.  However we can obtain full smoothness in $t$ by using the fact that $Af$ solves the heat equation and is already smooth in the spatial derivatives.
\end{proof}
\begin{cor} \label{cor:polyhomog2} If $f \in x^{\mu} C^{\infty}(\Mbar_T)$ and $H$ denotes the time convolution of the heat operator of $A$, then $Hf \in x^{\mu} C^{\infty}(\Mbar_T)$.
\end{cor}

We conclude this section with a proposition that we will need in the finer regularity analysis of the Ricci flow.  This shows that the commutator of a $b$-vector field with an element of the heat calculus remains in the calculus, and thus has the same mapping properties.
\begin{prop} \label{prop:commutator} If $A \in \Psi^{2,0}_{e;Heat}(X)$ and $V_b$ is any $b$-vector field, then
$[A, V_b] \in \Psi^{2,0}_{e;Heat}(X)$.
\end{prop}
\begin{proof}
The proof is similar to \cite[Proposition 3.30]{Mazzeo}, adapted to the heat calculus setting.  We sketch the proof here.  Return to the action on half-densities, equation \eqref{eqn:actionhalf}, and suppose that $f$ is a smooth half-density vanishing to all orders at the boundary hypersurfaces.  Now suppose for simplicity that $V_b = \partial_y$ is a $b$-vector field.  After an integration by parts, we may write 
\begin{align*}  & (\partial_y A f)(x,y,t) - (A \partial_y f)(x,y,t) \\ &= (\beta_L)_*\left( \left( \beta_L^*(\partial_y) + \beta_R^* (\partial_{y'}^T )\right) \rho_{00,2}^{-\frac{n}{2}} \rho_{11,0}^{-\frac{n+1}{2}} k \nu \cdot (\beta_R)^*( f(x',y') (x')^{-\frac{n+1}{2}} |dx' dy'|^{1/2} ) \right),
\end{align*}
where $\partial_{y'}^T$ is the adjoint of $\partial_{y'}$ under the measure.  The key now is that while each of $\partial_y$ and $\partial_{y'}$ lifts to a vector field singular near $B_{00,2}$, their sum cancels this behaviour.  Indeed, computing in the coordinates defined in equation \eqref{coord-near-tff}, we find that 
\[ \beta_L^* (\partial_y ) = \frac{1}{x' \tau} \partial_U \]
\[ \beta_R^* (\partial_{y'}^T) = -\frac{1}{x' \tau} \partial_U + \partial_{y'} + \; \mbox{smooth function}. \]
Consequently, $\beta_L^*(\partial_y) + \beta_R^* (\partial_{y'}^T)$ does not affect the asymptotics of the kernel, and $[A, V_b] \in \Psi^{2,0}_{e;Heat}(X)$.
\end{proof}

We conclude with a discussion of the main existence theorem for the inhomogeneous Cauchy problem:
\begin{equation} \label{inhomo-cauchy}
 \left\{ \begin{array}{rl} 
(\partial_t - L) u(\zeta,t) & = f(\zeta,t) \\
u(\zeta,0) &= 0,\end{array} \right. 
\end{equation}
where $f \in x^{\mu} C^{a, \frac{a}{2}}_e(M_T)$ and $L$ is a second order uniformly degenerate elliptic operator with coefficients in $C_e^{a}$.  

\begin{theorem} \label{thm:inhomo}
Suppose $L$ is a second order uniformly degenerate elliptic operator with time-independent coefficients.  For every $f \in x^{\mu} C^{a, \frac{a}{2}}_e(M_T)$ there is a solution $u$ to \eqref{inhomo-cauchy} in $x^{\mu} C^{2+a, \frac{2+a}{2}}_e(M)$.  Moreover, $u$ satisfies the Schauder-type estimate
\begin{equation}
 ||u||_{x^{\mu} C^{2+a, \frac{2+a}{2}}_e(M_T)} \leq K ||f||_{x^{\mu} C^{a, \frac{a}{2}}_e(M_T)}. 
\end{equation}
\end{theorem}
\begin{proof}
By Duhamel's principle, a solution to \eqref{inhomo-cauchy} is given by
\begin{equation} \label{eqn:duhamel}
u(\zeta, t) = \int_0^t \int_M h(\zeta, \zeta', t-t') f(\zeta', t') \dvol(\zeta') dt',
\end{equation}
where $h$ is the heat kernel of the heat operator $e^{tL}$, where $e^{tL} \in \Psi^{2,0}_{e, Heat}.$

We now discuss the estimates.  The case for nonzero weight $\mu$ reduces to the unweighted estimate, as to solve the inhomogeneous problem with $u \in x^{\mu} C^{2+a,\frac{2+a}{2}}_e(M)$ amounts to solving
\[ (\partial_t - x^{-\mu} L x^{\mu}) u'(\zeta,t) = f'(\zeta,t) \]
for with $u'$ and $f'$ in appropriate unweighted spaces.  The kernel of the conjugated operator $x^{-\mu} L x^{\mu}$ has precisely the same asymptotics as the kernel of $L$, as may be seen by working in the coordinate systems \eqref{coord-near-lc} and \eqref{coord-near-tff}.  So it suffices to check the mapping properties when $\mu = 0$.

The strategy is now to cut up the space $HM_e^2$.  Consider a function $\phi$ equal to one in a tubular neighbourhood of $B_{00,2}$ and vanishing outside a slightly larger tubular neighbourhood.  We may write the heat kernel as 
\[ h = h_1 + h_2 := \phi h + (1-\phi) h. \]
To prove \eqref{est:schauder}, it will suffice to estimate both 
\[ u_1(\zeta, t) = \int_0^t \int_M h_1(\zeta, \zeta', t-t') f(\zeta', t') \dvol(\zeta') dt' \]
and
\[ u_2(\zeta, t) = \int_0^t \int_M h_2(\zeta, \zeta', t-t') f(\zeta', t') \dvol(\zeta') dt' \]

Regarding the estimate for $u_2$, we view $h_2$ as a polyhomogeneous distribution vanishing to infinite order at $B_{10,0}, B_{01,0}, B_{00,1}, B_{00,2}$ and smooth up to $B_{11,0}$.  The estimates are then checked in each of the coordinate systems  \eqref{coord-near-lc} and \eqref{coord-near-tff} in a lengthy but straightforward manner.  For example, under the coordinate change \eqref{coord-near-tff}, the $0$-vector field $x \partial_x$ lifts to $(\tau^{-1} + S) \partial_S$, consequently
\begin{align*}
(x \partial x)^2 u(x, y, t) &= \int_0^t \int_M (x \partial x)^2 h_2(\zeta, \zeta', t-t') f(\zeta', t') \dvol(\zeta') dt' \\
&= \int_0^t \int ((\tau^{-1} + S) \partial_S)^2 h_2(S, U, x', y', t-\tau) f(x', y', \tau) \dvol(x',y') d\tau.
\end{align*}
We can now estimate the $L^{\infty}$ norm and H\"older seminorm of $(x \partial x)^2 u(x, y, t)$ using the fact that $h_2$ is smooth and vanishes to infinite order in $\tau$ to absorb the apparent singular factor of $\tau$.  The same is true for the other derivatives.  We omit the precise estimation but refer the reader to \cite{BahuaudDrydenVertman} for similar estimation.

Regarding the estimate for $u_1$, recall the Whitney decomposition outlined in Section \ref{section:functionspaces}.  Let $\psi_i: B_T \rightarrow (B_i)_T$ be an affine map taking a `standard' parabolic cylinder over the ball with centre $(1,0)$ and radius 1 in the right half plane model of hyperbolic space to the ball with centre $(x_i, y_i)$ and radius $\frac{1}{2} x_i$:
\[ \psi_i( v, w, t) = ( x_i v, y_i + x_i w, t). \]

Consider the pullback of $u_1$ under $\psi_i$:
\begin{align*}
(\psi_i^*u)(v,w,t) &= u_1( x_i v, y_i + x_i w, t ) \\
&= \int_0^t \int h_1 \left( \frac{x_i v-x'}{x'(t-t')^{1/2}}, \frac{y_i + x_i w -y'}{x'(t-t')^{1/2}}, x', y', t-t'\right) \\
& \hspace{1 in} \cdot f(x',y',t') (x')^{-n-1} \dvol(x',y') dt' \\
&= \int_0^t \int h_1 \left( \frac{v-\vt}{\vt(t-t')^{1/2}}, \frac{w-\wt}{\vt(t-t')^{1/2}}, x_i \vt, y_i + x_i \wt, t-t'\right) \\
& \hspace{1 in} \cdot f(x_i \vt, y_i + x_i \wt,t') \vt^{-n-1} \dvol(\vt,\wt) dt'\\
&= \int_0^t \int h_1 \left( \frac{v-\vt}{\vt(t-t')^{1/2}}, \frac{w-\wt}{\vt(t-t')^{1/2}},  x_i \vt, y_i + x_i \wt, t-t'\right) \\
& \hspace{1 in} \cdot (\psi_i^* f)(\vt,\wt,t') \vt^{-n-1} \dvol(\vt,\wt) dt'
\end{align*}

These charts essentially provide the bookkeeping for rescaling of the equation to a fixed parabolic cylinder in $\bR^{n+1} \times [0, \infty)$.  Further, uniformly degenerate vector fields pull back under $\psi_i$ to differential operators with uniformly bounded coefficients on the standard cylinder, and in particular a uniformly degenerate parabolic operator is pulled back to a uniformly parabolic operator.  This reduces the estimate to the `classical' parabolic case, which may be found in \cite{Ladyzenskaja}.

\end{proof}

\bibliographystyle{amsplain}
\bibliography{rflow-cc}

\def\cprime{$'$} \def\cprime{$'$}
\providecommand{\bysame}{\leavevmode\hbox to3em{\hrulefill}\thinspace}
\providecommand{\MR}{\relax\ifhmode\unskip\space\fi MR }
\providecommand{\MRhref}[2]{%
  \href{http://www.ams.org/mathscinet-getitem?mr=#1}{#2}
}
\providecommand{\href}[2]{#2}
\begin{thebibliography}{10}

\bibitem{Albin}
Pierre Albin, \emph{A renormalized index theorem for some complete
  asymptotically regular metrics: the {G}auss-{B}onnet theorem}, Adv. Math.
  \textbf{213} (2007), no.~1, 1--52. \MR{MR2331237 (2008h:58043)}

\bibitem{AAR-cusp}
Pierre Albin, Clara~L. Aldana, and Fr\'ed\'eric Rochon, \emph{Ricci flow and
  the determinant of the {L}aplacian on non-compact surfaces},  (2009).

\bibitem{BahuaudDrydenVertman}
Eric Bahuaud, Emily Dryden, and Boris Vertman, \emph{Mapping properties of the
  heat operator on edge manifolds}, 2011.

\bibitem{BahuaudGicquaud}
Eric Bahuaud and Romain Gicquaud, \emph{Conformal compactification of
  asymptotically locally hyperbolic metrics}, Journal of Geometric Analysis
  (2010).

\bibitem{BahuaudHelliwell}
Eric Bahuaud and Dylan Helliwell, \emph{Short-time existence for some
  higher-order geometric flows}, Comm. Partial Differential Equations (2011).

\bibitem{Bamler}
Richard Bamler, \emph{Stability of hyperbolic manifolds with cusps under
  {R}icci flow},  (2010).

\bibitem{Biquard}
Olivier Biquard (ed.), \emph{Ad{S}/{CFT} correspondence: {E}instein metrics and
  their conformal boundaries}, IRMA Lectures in Mathematics and Theoretical
  Physics, vol.~8, European Mathematical Society (EMS), Z\"urich, 2005, Papers
  from the 73rd Meeting of Theoretical Physicists and Mathematicians held in
  Strasbourg, September 11--13, 2003. \MR{2160864 (2006b:53001)}

\bibitem{ChenZhu}
Bing-Long Chen and Xi-Ping Zhu, \emph{Uniqueness of the {R}icci flow on
  complete noncompact manifolds}, J. Differential Geom. \textbf{74} (2006),
  no.~1, 119--154. \MR{MR2260930 (2007i:53071)}

\bibitem{Chow}
Bennett Chow, Peng Lu, and Lei Ni, \emph{Hamilton's {R}icci flow}, Graduate
  Studies in Mathematics, vol.~77, American Mathematical Society, Providence,
  RI, 2006. \MR{MR2274812 (2008a:53068)}

\bibitem{GrahamLee}
C.~Robin Graham and John~M. Lee, \emph{Einstein metrics with prescribed
  conformal infinity on the ball}, Adv. Math. \textbf{87} (1991), no.~2,
  186--225. \MR{1112625 (92i:53041)}

\bibitem{Grieser}
Daniel Grieser, \emph{Basics of the {$b$}-calculus}, Approaches to singular
  analysis ({B}erlin, 1999), Oper. Theory Adv. Appl., vol. 125, Birkh\"auser,
  Basel, 2001, pp.~30--84. \MR{MR1827170 (2002e:58051)}

\bibitem{HQS}
Xue Hu, Jie Qing, and Yuguang Shi, \emph{Regularity and rigidity of
  asymptotically hyperbolic manifolds}, 2009.

\bibitem{IMS-ac}
James Isenberg, Rafe Mazzeo, and Natasha Sesum, \emph{Ricci flow on
  asymptotically conical surfaces with nontrivial topology},  (2010).

\bibitem{JMN-cusp}
Lizhen Ji, Rafe Mazzeo, and Natasa Sesum, \emph{Ricci flow on surfaces with
  cusps}, Math. Ann. \textbf{345} (2009), no.~4, 819--834. \MR{MR2545867}

\bibitem{KochLamm}
Herbert Koch and Tobias Lamm, \emph{Geometric flows with rough initial data},
  (2009).

\bibitem{Krylov}
N.~V. Krylov, \emph{Lectures on elliptic and parabolic equations in {H}\"older
  spaces}, Graduate Studies in Mathematics, vol.~12, American Mathematical
  Society, Providence, RI, 1996. \MR{MR1406091 (97i:35001)}

\bibitem{Ladyzenskaja}
O.~A. Lady{\v{z}}enskaja, V.~A. Solonnikov, and N.~N. Ural{\cprime}ceva,
  \emph{Linear and quasilinear equations of parabolic type}, Translated from
  the Russian by S. Smith. Translations of Mathematical Monographs, Vol. 23,
  American Mathematical Society, Providence, R.I., 1967. \MR{MR0241822 (39
  \#3159b)}

\bibitem{LeeSpectrum}
John~M. Lee, \emph{The spectrum of an asymptotically hyperbolic {E}instein
  manifold}, Comm. Anal. Geom. \textbf{3} (1995), no.~1-2, 253--271.
  \MR{1362652 (96h:58176)}

\bibitem{Lee}
\bysame, \emph{Fredholm operators and {E}instein metrics on conformally compact
  manifolds}, Mem. Amer. Math. Soc. \textbf{183} (2006), no.~864, vi+83.
  \MR{MR2252687 (2007m:53047)}

\bibitem{LiYin}
Haozhao Li and Hao Yin, \emph{On stability of the hyperbolic space form under
  the normalized {R}icci flow}, Int. Math. Res. Not. IMRN (2010), no.~15,
  2903--2924. \MR{2673714}

\bibitem{MaXu}
Li~Ma and Xingwang Xu, \emph{Ricci flow with hyperbolic warped product
  metrics}, arXiv preprint (2007).

\bibitem{Mazzeo}
Rafe Mazzeo, \emph{Elliptic theory of differential edge operators. {I}}, Comm.
  Partial Differential Equations \textbf{16} (1991), no.~10, 1615--1664.
  \MR{MR1133743 (93d:58152)}

\bibitem{MazzeoYamabe}
\bysame, \emph{Regularity for the singular {Y}amabe problem}, Indiana Univ.
  Math. J. \textbf{40} (1991), no.~4, 1277--1299. \MR{MR1142715 (92k:53071)}

\bibitem{MazzeoPacard}
Rafe Mazzeo and Frank Pacard, \emph{Maskit combinations of
  {P}oincar\'e-{E}instein metrics}, Adv. Math. \textbf{204} (2006), no.~2,
  379--412. \MR{MR2249618 (2007e:53052)}

\bibitem{MazzeoTaylor}
Rafe Mazzeo and Michael Taylor, \emph{Curvature and uniformization}, Israel J.
  Math. \textbf{130} (2002), 323--346. \MR{1919383 (2003j:30063)}

\bibitem{Melrose}
Richard~B. Melrose, \emph{The {A}tiyah-{P}atodi-{S}inger index theorem},
  Research Notes in Mathematics, vol.~4, A K Peters Ltd., Wellesley, MA, 1993.
  \MR{MR1348401 (96g:58180)}

\bibitem{Woolgar}
Todd~A. Oliynyk and Eric Woolgar, \emph{Rotationally symmetric {R}icci flow on
  asymptotically flat manifolds}, Comm. Anal. Geom. \textbf{15} (2007), no.~3,
  535--568. \MR{2379804 (2009b:53110)}

\bibitem{QSW}
Jie Qing, Yuguang Shi, and Jie Wu, \emph{Normalized {R}icci flow and
  conformally compact {E}instein metrics}, 2011.

\bibitem{SSS}
Oliver~C. Schn{\"u}rer, Felix Schulze, and Miles Simon, \emph{Stability of
  {H}yperbolic space under {R}icci flow}, 2010.

\bibitem{Shi}
Wan-Xiong Shi, \emph{Deforming the metric on complete {R}iemannian manifolds},
  J. Differential Geom. \textbf{30} (1989), no.~1, 223--301. \MR{MR1001277
  (90i:58202)}

\end{thebibliography}
\end{document}